\documentclass[final]{siamltex}

\usepackage{amsmath,amssymb,hyperref,latexsym,comment}
\usepackage{xcolor}
\hypersetup{colorlinks}
\usepackage{graphicx}
\usepackage{epsf}

\usepackage{amsfonts}
\usepackage{graphics}
\usepackage{epsfig}
\usepackage{float}

\usepackage{enumerate}

\usepackage{algorithm, algorithmic}
\usepackage{subfig}
\usepackage{multirow}
\usepackage{threeparttable}
\usepackage{subeqnarray}

\usepackage{tikz-cd}

\usepackage{cases}

\usepackage{chngcntr}
\usepackage{url,xcolor}
\usepackage{mwe,tikz}
\usepackage[percent]{overpic}
\usepackage{tikz}
\usepackage{pgfplots}
\usetikzlibrary{spy,calc}
\usepackage{hyperref}
\usepackage{adjustbox}
\usepackage{makecell}

\usepackage{mathrsfs}

\counterwithout{figure}{section}
\counterwithout{table}{section}

\newcommand{\twocm}[1]{#1} 
\newcommand{\onecm}[1]{}

\newtheorem{thm}{Theorem}

 \newtheorem{assumption}{Assumption}

 \newtheorem{rem}[theorem]{Remark}
 \numberwithin{equation}{section}

\newcommand{\beq}{\begin{equation}}
\newcommand{\eeq}{\end{equation}}

\newcommand{\cV}{{\mathcal{V}}}

\newcommand{\argmin}{\mathop{\mathrm{argmin}}}

\newcommand{\abs}[1]{\, {\left\vert #1\right\vert}\, }

\newcommand{\myatop}[2]{\genfrac{}{}{1.5pt}{}{#1}{#2}}

\usepackage{xcolor}
\definecolor{newcolor}{rgb}{.8,.349,.1}

\newcommand{\RN}[1]{%
  \textup{\uppercase\expandafter{\romannumeral#1}}%
}

\title{Bilinear constraint based ADMM for mixed Poisson-Gaussian noise removal}

\author{
     Jie Zhang\thanks{School of Mathematical Sciences, Tianjin Normal University}
\and Yuping Duan\thanks{Center for Applied Mathematics, Tianjin  University}
\and Yue Lu\thanks{School of Mathematical Sciences, Tianjin Normal University}
\and Michael K. Ng \thanks{Department of Mathematics, The University of Hong Kong}
\and Huibin Chang\thanks{Corresponding author. School of Mathematical Sciences, Tianjin Normal University}
}

\begin{document}

% Title
\maketitle
\begin{abstract}
In this paper, we propose new operator-splitting algorithms for the total variation regularized infimal convolution (TV-IC) model \cite{article11}  in order to remove mixed Poisson-Gaussian (MPG) noise. In the existing splitting algorithm for TV-IC, an inner loop by Newton method had to be adopted for one nonlinear optimization subproblem, which increased the computation cost per outer loop. By introducing a new bilinear constraint and applying the alternating direction method of multipliers (ADMM), all subproblems of the proposed algorithms named as BCA (short for {\bf B}ilinear {\bf C}onstraint based {\bf A}DMM algorithm) and BCA$_{f}$ (short for a variant of BCA with ${\bf f}$ully splitting form) can be very efficiently solved; especially for the proposed  BCA$_{f}$, they can be calculated without any inner iterations. Under mild conditions, the convergence of the proposed BCA is investigated. Numerically, compared to existing primal-dual algorithms for the TV-IC model, the proposed algorithms, with fewer tunable parameters, converge much faster and produce comparable results meanwhile.
\end{abstract}


\section{Introduction}
As the result of photon counting and thermal noise to the detectors, it is very common that the observed image is corrupted by the mixed Poisson and Gaussian (MPG) noise. 
The MPG denoising  has been extensively studied in \cite{article15,article16,article19,article20,article21,article22} and references therein. 
Generally speaking, the idea of MPG noise removal is based on the maximum a posteriori (MAP) following the Bayes' law. Chakrabarti and Zickler \cite{article3} approximated the MPG noise with a shifted Poisson likelihood. A generalized Anscombe transformation was proposed for MPG noise removal in \cite{Starck:1998:IPD:289385,article18}, while its unbiased inversion was given in \cite{article17}. In order to choose a correct MPG noise model, Reyes and Sch\"onlieb \cite{article12} proposed a nonsmooth PDE-constrained optimization strategy. A reweighted $L^2$ method was proposed by Li et al. \cite{article9}, which approximated the Poisson component noise with weighted Gaussian noise. The convexity and Lipschitz differentiability of Poisson-Gaussian negative log-likelihood was proven by Chouzenoux et al. in \cite{article8}, where a convergent primal-dual algorithm was given in the case of approximation of the infinite sum for data discrepancy.

More recent works considered the general joint MAP formulation, showing that Gaussian noise model \cite{article6} and Poisson noise model \cite{article4} can be combined together in order to remove the MPG noise. 
Lanza et al. \cite{article10} proposed a primal-dual based iterative algorithm for total variation regularized model (TV-PD), where one subproblem required additional inner loop by Newton method. In practice, in order to reduce computational cost, the TV-PD algorithm ran with very few Newton iterations and the corresponding convergence guarantee with such inexact inner solver was unknown. Calatroni et al. \cite{article11} proposed {the total variation regularized infimal convolution} (TV-IC) model, consisting of infimal convolution combination of standard data fidelities classically associated to one single-noise distribution, and a total variation (TV) regularization, which is essentially an extension of  {\cite{article10}} by relaxing the relations of two data fitting terms. A semi-smooth Newton algorithm was proposed in \cite{article11} with the weak singularity of the first order derivative and high dimension linear systems. 

{In order to solve the TV-IC model more efficiently, we will introduce a new bilinear constraint to reformulate the model, which essentially  helps to establish the iterative algorithms  without Newton iteration in the inner loop.  Then we apply alternating direction method of multipliers (ADMM)  \cite{glowinski1975approximation,gabay1976dual,eckstein1992douglas,ADMM1,wu2011augmented} to the reformulated model, leading to the proposed {\bf B}ilinear {\bf C}onstraint based {\bf A}DMM algorithm (BCA). 
% In order to solve the TV-IC model more efficiently, we will introduce a new bilinear constraint to reformulate the model, and then apply alternating direction method of multipliers (ADMM)  \cite{glowinski1975approximation,gabay1976dual,eckstein1992douglas,ADMM1,wu2011augmented} to the reformulated
% model, that leads to the proposed {\bf B}ilinear {\bf C}onstraint based {\bf A}DMM algorithm (BCA). 
%Namely, the subproblem of all auxiliary variables will have closed form solutions, that enables much faster convergence. 
}
Due to the nonconvex term in the augmented Lagrangian caused by the bilinear constraint, it seems quite difficult to study the global convergence. Instead, by assuming the iterative sequences have a uniform and strictly positive lower bounds, we prove the local convergence of proposed BCA, in the sense that corresponding iterative sequences are bounded and any limit point is a stationary point of the saddle problem of the augmented Lagrangian functional. 
Due to the existence of total variation term of the original variable, inner loop is still needed for the proposed BCA. In order to reduce such extra computational cost, we further develop a variant of {\bf BCA} with {\bf f}ully splitting form (BCA$_f$). Extensive numerical experiments further verify the faster convergence of the proposed algorithms while producing comparable recovery results. Especially, it demonstrates that the proposed algorithms with fewer tunable parameters converge much faster  than  TV-PD  \cite{article10} for the TV-IC model.   
%In summary, the main contributions of  this paper consists of (1) Introducing two novel constraints such that the proposed ADMM algorithms for TV-IC model (BCA and BCA$_f$) are very efficiently solved, and (2) the convergence analysis of the proposed BCA will also be given.

This paper is organized as follows. Section 2 reviews briefly the TV-IC model and ADMM algorithm. Section 3 introduces the proposed BCA and BCA$_f$ algorithms, and the convergence analysis of BCA to the stationary points is also provided. Section 4 presents the numerical experiments to show the performances of proposed algorithms in terms of convergence and recovery quality as well as the robustness with respect to the parameters and the number of inner iterations of BCA. Finally, conclusions and future work are given in section 5.

\section{Review of TV-IC Model and ADMM}
\label{sec:review}
In this section, we will briefly review the TV-IC model  and the ADMM algorithm.
\subsection{Review of TV-IC Model}
Let $u, f\in \mathbb{R}^n$ be the ground truth and observed images corrupted by MPG noise respectively, satisfying  
\[f=v+\mathbf  n,\]
where $v\sim \mathrm{Poisson}(u)$, $\mathbf  n\sim \mathcal N (0,\sigma^{2})$.
The general joint MAP estimation \cite{article10,article11} is given 
\begin{equation}
\label{eq:1}
\begin{aligned}
    ({u^\star},{v^\star})&=\arg\max_{(u,v)}\prod\limits_{i}  P(v_i,u_i|f_i)\\
                       &=\arg\max_{(u,v)} \prod\limits_{i}  P(f_i|v_i)P(v_i|u_i)P(u_i).
\end{aligned}
\end{equation}
Incorporating the density function of Poisson and Gaussian distributions, and further taking the negative logarithm in \eqref{eq:1}, one has
\begin{displaymath}
\begin{aligned}
(u^\star,v^\star)&=\arg\min\limits_{(u,v)} -\ln(\prod\limits_{i} P(f_i|v_i)P(v_i|u_i)P(u_i))\\
&=\arg\min\limits_{(u,v)}\left\{\frac{\lambda_{1}}{2}\sum\limits_{i}(f_{i}-v_{i})^{2}+\sum\limits_{i}|\nabla u_{i}|+\lambda_{2}\sum\limits_{i}(u_i-v_i\ln u_i+\ln {v_i!})\right\},
\end{aligned}
\end{displaymath}
with $\nabla$ denoting the discrete gradient (finite difference) operator and $|\cdot|$ denotes the $L^2$ norm of a vector, 
where a Gibbs prior distribution of $P(u_i)=\exp(-|\nabla u_i|)$ is considered.

Using the standard Stirling approximation of the logarithm of the factorial function, the following TV-IC model was established \cite{article10,article11}
\begin{equation}
\label{eq:2}
\min\limits_{u,v} H(u,v),
\end{equation}
where $H(u,v)$ is:
\begin{displaymath}
H(u,v)=\frac{\lambda_{1}}{2}\sum\limits_{i}(f_{i}-v_{i})^{2}+\lambda_{2}\sum\limits_{i}(u_i-{ v_i\ln\frac{ u_i}{v_i}}-v_i)+\sum\limits_{i}\vert\nabla u_{i}\vert+\chi_{\mathcal{V}}(v),
\end{displaymath}
$\chi_{\cV}$ is the characteristic function of the positivity constraint set $\cV=\{v:v_{i}\geq \epsilon> 0~\forall i\}$ 
\begin{displaymath}
\chi_{\cV}(v)=\left\{\begin{array}{ll}0&~v\in \cV,\\
+\infty&\text{~otherwise,}
\end{array}\right.
\end{displaymath}
which is derived by the property of $v$ ($v\sim \mathrm{Poisson}(u)$).  
The notation $\sum\limits_{i}|\nabla u_i|$ denotes the standard discrete TV regularization. 
%The data fidelity of this model has the following infimal convolution structure  \cite{article11}:
%\begin{displaymath}
%\Phi(u, f):=\inf_{v}\left\{\mathscr{F}(f, u, v):= \Phi_{1}(v)+ \Phi_{2}(u, f-v)\right\},
%\end{displaymath}
%which is the reason why it is called TV-IC.
%}
Here we remark that we require that $v$ is lower bounded by a positive scalar $\epsilon$, which is introduced for the purpose of studying  convergence guarantee of proposed algorithms.

\subsection{Review of ADMM} \label{secReviewADMM}
The ADMM \cite{glowinski1975approximation,gabay1976dual,eckstein1992douglas,ADMM1,wu2011augmented,wang2015global,Mei2018} is one of the popular first-order operator-splitting algorithm in image processing, which can handle complex constraints and non-smooth and non-convex objective functional. Compared with the gradient descent algorithm, it is more stable since it gets rid of directly calculating the derivative of the objective functional and therefore allows for big stepsize. Hence, we apply the ADMM to solve the TV-IC model. In this part, we will give a brief introduction of the ADMM. Consider the optimization problem below
\begin{equation}
\label{eq:3}
\begin{aligned}
\begin{array}{ll}{\min\limits_{\mathbf x, \mathbf z}} & {f(\mathbf x)+g(\mathbf z)} \\ 
{\text { s.t. }} & {A \mathbf x+B \mathbf z=\mathbf m}
\end{array}
\end{aligned}
\end{equation}
with variables $\mathbf x \in \mathbb{R}^{n}$ and $\mathbf z \in \mathbb{R}^{m}$, where $A \in \mathbb{R}^{p \times n}$, $B \in \mathbb{R}^{p \times m}$, and $\mathbf m \in \mathbb{R}^{p}$. 
The augmented Lagrangian is given below
\begin{displaymath}
L_{\rho}(\mathbf x, \mathbf z, \mathbf y)= f(\mathbf x)+g(\mathbf z)+\mathbf y^{T}(A \mathbf x+B \mathbf z-\mathbf m)+\frac{\rho}{2}\left\|A \mathbf x+B \mathbf z-\mathbf m\right\|_{2}^{2},
\end{displaymath}
with the parameter $\rho>0$ and the multiplier $\mathbf y\in\mathbb R^p$.
In order to solve the saddle point problem 
\[
\max_{\mathbf y}\min_{\mathbf x,\mathbf z} L_{\rho}(\mathbf x, \mathbf z, \mathbf y),
\]
The ADMM consists of the following iterations to determine $(k+1)^{th}$ solutions as 
\begin{displaymath}
\begin{aligned}
&\mathbf x^{k+1} :=\argmin_{\mathbf x} L_{\rho}\left(\mathbf x, \mathbf z^{k}, \mathbf y^{k}\right)\\
&\mathbf z^{k+1} :=\argmin_{\mathbf z} L_{\rho}\left(\mathbf x^{k+1}, \mathbf z, \mathbf y^{k}\right)\\
&\mathbf y^{k+1} :=\mathbf y^{k}+\rho\left(A \mathbf x^{k+1}+B \mathbf z^{k+1}-\mathbf m\right)，
\end{aligned}
\end{displaymath}
given the previous iteration solutions $(\mathbf x^k, \mathbf z^k, \mathbf y^k)$. 

\section{Proposed Algorithms}
\label{sec:algorithm}
In this section, we will consider how to design more efficient operator-splitting algorithms  based on ADMM  to solve the TV-IC model \eqref{eq:2}. If decoupling the problem by introducing an auxiliary variable to replace the original variable $u$ following \cite{article10}, it will have a subproblem without closed-form solution, due to the existence of the term $v_i\ln v_i$. In order to solve  $v-$subproblem, an inner loop by Newton method is needed, which is time-consuming and lack of convergence guarantee with {\color{red}a} few inner iterations. To further speed up the convergence, a new bilinear constraint ($u_i=v_iw_i$) will be introduced such that the resulting BCA  algorithm consists of standard TV-L2 denoising for variable $u$,  and simple closed form solutions for $v$ and $w$.  Its convergence is further derived under the assumption that the iterative sequence of $w$ is uniformly bounded below by a positive number. In order to get a fully splitting scheme, i.e. all subproblems have closed form solutions, a typical constraint $p_i=\nabla u_i$ is introduced additionally such that  the BCA$_f$ is obtained within the framework of ADMM.  Although we can not prove its theoretical convergence following the technique developed for BCA, it converges well numerically as demonstrated in the numerical section. 

\subsection{BCA}

By introducing the bilinear constraint $u_i = v_i w_i$, we rewrite \eqref{eq:2} as the following equivalent constrained optimization problem:
\begin{equation}
\label{eq:4}
\begin{aligned}
&\min\limits_{u,v, w}\left\{ \frac{\lambda_{1}}{2}\sum\limits_{i}(f_{i}-v_{i})^{2}+\lambda_{2}\sum\limits_{i}(u_i-v_i\ln w_i-v_i)+\sum\limits_{i}\vert \nabla u_{i}\vert+\chi_{\mathcal{V}}(v)\right\},\\
&s.t.\ \ \ \ \ \ \  u_i=v_i w_i,\  \forall 1\leq i\leq n.
\end{aligned}
\end{equation}
Readily one sees that the term $v_i\ln v_i$ disappears, hence we can design a fast algorithm with $v$ and $w$ subproblems all having closed form solutions. 
As  reviewed in subsection \ref{secReviewADMM}, one has to establish the augmented Lagrangian of the above constrained optimization problem with the penalization parameter $\alpha>0$ and the multiplier $\Lambda$, is given below
\begin{equation}
\label{eq:5}
\begin{aligned}
L_{\alpha}(u,v,w,\Lambda)=& \frac{\lambda_{1}}{2}\sum\limits_{i}(f_{i}-v_{i})^{2}+\sum\limits_{i}\vert \nabla u_{i}\vert+\chi_{\mathcal{V}}(v)\\
&+\lambda_{2}\sum\limits_{i}(u_i-v_i\ln w_i-v_i)\\
&+\langle \Lambda,v\circ w-u\rangle+\frac{\alpha}{2}\Vert v\circ w-u\Vert^2,
\end{aligned}
\end{equation}
where $\left\langle\cdot\right\rangle$ and $\left\|\cdot\right\|$ denote the inner product and norm in $L^2$ space {respectively}, and $\circ$ denotes the element-wise multiplication. Note that all the vector multiplications and divisions in this paper are element-wise.

Given the previous iterative solution $(u^k,v^k,w^k)$, the ADMM updates the sequence $(u^{k+1},v^{k+1},w^{k+1})$ by solving three subproblems w.r.t. $u,$ $v,$ $w$, and multiplier update, which is given below:
\begin{subequations}
\begin{numcases}{}
u^{k+1}=\argmin_{u}L_{\alpha}(u,v^k,w^k,\Lambda^k),\label{eq:6}\\
v^{k+1}=\argmin_v L_{\alpha}(u^{k+1},v,w^k,\Lambda^k),\label{eq:7}\\
w^{k+1}=\argmin_w L_{\alpha}(u^{k+1},v^{k+1},w,\Lambda^k),\label{eq:8}\\
\Lambda^{k+1}=\Lambda^k+\alpha(v^{k+1}\circ w^{k+1}-u^{k+1}).\label{eq:9}
\end{numcases}
\end{subequations}
We will show how to solve these subproblems in the rest of this part.

First, we consider the $u$-subproblem as 
\begin{equation}
\label{eq:10}
\begin{aligned}
u^{k+1}&=\argmin_{u}\left\{\lambda_2\sum\limits_i u_i+\langle\Lambda^{k},v^{k}\circ w^{k}-u\rangle+\frac{\alpha}{2}\left\| v^{k}\circ w^{k}-u\right\|^{2}+\sum\limits_i\abs{\nabla u_i}\right\}\\
&=\argmin_{u}\left\{\frac{\alpha}{2}\left\|v^k\circ w^k+\frac{\Lambda^k}{\alpha}-\frac{\lambda_2\mathbf{1}}{\alpha}-u\right\|^2+\sum\limits_i\abs{\nabla u_i} \right\},
\end{aligned}
\end{equation}
where $\mathbf{1}\in \mathbb R^n$ is a vector whose elements are all equal to one. This is a standard TV-L2 \cite{article6} optimization problem, and one can adopt the gradient projection algorithm for the pre-dual form of total variation minimization \cite{Chambolle:2004:ATV:964969.964985}.

{
As the update rule of $w$-subproblem can simplify the calculation of $v$-subproblem, we consider $w$-subproblem first.
\begin{displaymath}
w^{k+1}=\argmin_{w}\sum\limits_{i}\left[-\lambda_2 v_i^{k+1}\ln w_i+\frac{\alpha}{2}(v_{i}^{k+1} w_i+\frac{\Lambda_{i}^k}{\alpha}-u_i^{k+1})^2\right].
\end{displaymath}
Obviously it is a convex optimization problem. One can readily get the scalar optimization problem of this convex optimization problem 
\begin{displaymath}
{w_i^{k+1}=\arg\min_{w_i}\left\{-\lambda_2v_i^{k+1} \ln w_i+\frac{\alpha}{2}(v_i^{k+1} w_i+\frac{\Lambda_i^k}{\alpha}-u_i^{k+1})^2 \right\}.}
\end{displaymath}
The optimality condition of the above problem is
\begin{displaymath}
\alpha(v_i^{k+1})^2w_i^2+\left(\Lambda_i^k v_i^{k+1}-\alpha v_i^{k+1} u_i^{k+1}\right)w_i-\lambda_2v_i^{k+1}=0.
\end{displaymath}
We can obtain a closed-form solution (also the global minimizer) of this problem
\begin{equation}
\label{eq:12}
w_i^{k+1}=\frac{1}{2v_i^{k+1}}\left[(u_i^{k+1}-\frac{\Lambda_{i}^k}{\alpha})+\sqrt{(u_i^{k+1}-\frac{\Lambda_{i}^k}{\alpha})^2+\frac{4\lambda_2 v_i^{k+1}}{\alpha}}\ \right].
\end{equation}

Finally, we consider the $v$-subproblem as 
\begin{displaymath}
v^{k+1}=\argmin_{v_i\ge\epsilon}\sum\limits_{i}\left[\frac{\lambda_1}{2}(f_i-v_i)^2-\lambda_2(v_i\ln w_i^k+v_i)+\frac{\alpha}{2}(v_i w_i^k+\frac{\Lambda_{i}^k}{\alpha}-u_i^{k+1})^2\right].
\end{displaymath}
This optimization problem can be computed independently with respect to each component of $v$, therefore, we can consider the scalar optimization problem
\begin{displaymath}
v_i^{k+1}=\argmin_{v_i\ge \epsilon}\left\{\frac{\lambda_1}{2}(f_i-v_i)^2-\lambda_2(v_i\ln w_i^{k}+v_i)+\frac{\alpha}{2}(v_i w_i^{k}+\frac{\Lambda_i^k}{\alpha}-u_i^{k+1})^2\right\}.
\end{displaymath}
One easily obtains the optimal solution of the above problem below
\begin{displaymath}
v^{k+1}=\max\left(\epsilon\mathbf 1,\tilde{v}^{k+1}\right),
\end{displaymath}
with the notation $\max(\cdot,\cdot)$ taking the {element-wise} maximum of two vectors,  where $\tilde{v}^{k+1}$ is defined as follows
\begin{equation}
\label{eq:11}
\tilde{v}^{k+1}=\myatop{\mathbf{1}}{\lambda_1\mathbf{1}+\alpha(w^k)^2}\circ\left(\lambda_1 f+\lambda_2\ln w^k+\lambda_2-w^k\circ\Lambda^k+\alpha w^k\circ u^{k+1}\right),
\end{equation}
which corresponds to the unconstrained optimal solution. Note that $\myatop{\mathbf{a}}{\mathbf{b}}$ denotes the element-wise division of two vectors $\mathbf{a}$ and $\mathbf{b}$.

% Finally we consider the $w$-subproblem below:
% \begin{displaymath}
% w^{k+1}=\arg\min_{w}\sum\limits_{i}\left[-\lambda_2 v_i^{k+1}\ln w_i+\frac{\alpha}{2}(v_{i}^{k+1} w_i+\frac{\Lambda_{i}^k}{\alpha}-u_i^{k+1})^2\right].
% \end{displaymath}
% Obviously, this is a convex optimization problem. One can readily get the scalar optimization problem of this convex optimization problem 
% \begin{displaymath}
% {w_i^{k+1}=\arg\min_{w_i}\left\{-\lambda_2v_i^{k+1} \ln w_i+\frac{\alpha}{2}(v_i^{k+1} w_i+\frac{\Lambda_i^k}{\alpha}-u_i^{k+1})^2 \right\}.}
% \end{displaymath}
% The optimality condition of the above problem is
% \begin{displaymath}
% \alpha(v_i^{k+1})^2w_i^2+\left(\Lambda_i^k v_i^{k+1}-\alpha v_i^{k+1} u_i^{k+1}\right)w_i-\lambda_2v_i^{k+1}=0.
% \end{displaymath}
% We can obtain a closed-form solution of this problem
% \begin{equation}
% \label{eq:12}
% w_i^{k+1}=\frac{1}{2v_i^{k+1}}\left[(u_i^{k+1}-\frac{\Lambda_{i}^k}{\alpha})+\sqrt{(u_i^{k+1}-\frac{\Lambda_{i}^k}{\alpha})^2+\frac{4\lambda_2 v_i^{k+1}}{\alpha}}\ \right].
% \end{equation}

In order to further simplify the calculation of $v-$subproblem, one can obtain the following lemma.
\begin{lemma}
\label{le:1}
Letting $\Lambda^{k+1}, w^{k+1}$ be generated by \eqref{eq:6}-\eqref{eq:9}, then  we have
\begin{equation}
\label{eq:13}
\Lambda^{k+1}\circ w^{k+1}=\lambda_2\mathbf{1}.
\end{equation}
\end{lemma}
\begin{proof}
Considering the update rule of $w$-subproblem in \eqref{eq:8} and multiplier update in \eqref{eq:9}, we have
\begin{displaymath}
\begin{aligned}
0&=\alpha(v^{k+1})^2\circ(w^{k+1})^2+(\Lambda^k\circ v^{k+1}-\alpha v^{k+1}\circ u^{k+1})\circ w^{k+1}-\lambda_2 v^{k+1}\\
&=v^{k+1}\circ w^{k+1}\circ(\Lambda^k+\alpha v^{k+1}\circ w^{k+1}-\alpha u^{k+1})-\lambda_2v^{k+1}\\
&=v^{k+1}\circ w^{k+1}\circ\Lambda^{k+1}-\lambda_2 v^{k+1}.\\
\end{aligned}
\end{displaymath}
Further due to $v_i^{k+1}\ge\epsilon>0\ \ \forall ~ i$, we can prove this lemma.
\end{proof}

\begin{rem}
\label{rem:1}
By Lemma \ref{le:1}, The equation \eqref{eq:11} can be  simplified below
\begin{displaymath}
\begin{aligned}
\tilde{v}^{k+1}&=\myatop{\mathbf{1}}{\lambda_1\mathbf{1}+\alpha(w^k)^2}\circ\left(\lambda_1 f+\lambda_2\ln w^k+\lambda_2-w^k\circ\Lambda^k+\alpha w^k\circ u^{k+1}\right),\\
&=\myatop{\mathbf{1}}{\lambda_1\mathbf{1}+\alpha(w^k)^2}\circ\left(\lambda_1 f+\lambda_2\ln w^k+\alpha w^k\circ u^{k+1}\right).
\end{aligned}
\end{displaymath}
Therefore,
\begin{equation}
\label{eq:14}
\begin{aligned}
v^{k+1}&=\max\left(\epsilon \mathbf 1,\tilde{v}^{k+1}\right),\\
&=\max\left(\epsilon \mathbf 1,\myatop{\mathbf{1}}{\lambda_1\mathbf{1}+\alpha(w^k)^2}\circ\left(\lambda_1 f+\lambda_2\ln w^k+\alpha w^k\circ u^{k+1}\right)\right).
\end{aligned}
\end{equation}
\end{rem}
}
Algorithm 1 summarizes the overall BCA algorithm.
\begin{algorithm}[H]
	\renewcommand{\algorithmicrequire}{\textbf{Input:}}
	\renewcommand{\algorithmicensure}{\textbf{Initialization:}}
	\caption{BCA}
	\label{alg:1}
	\begin{algorithmic}[1]
		\REQUIRE {Noisy data $f$ and parameters $\lambda_1$, $\lambda_2$, $\alpha$}\\
		\ENSURE {{ $u^{0}=f, v^{0}=f, w^0=\mathbf 1, \Lambda^{0}=\mathbf 0, k=0.$}}\\
		\WHILE{Stopping criteria is not satisfied}
		\STATE Solve  $u^{k+1}$ by \eqref{eq:10}
		\STATE Solve  $v^{k+1}$ by \eqref{eq:14}
		\STATE Solve  $w^{k+1}$ by \eqref{eq:12}
		\STATE 
		Update the multipliers by 
		\[
		\Lambda^{k+1}=\Lambda^{k}+\alpha(v^{k+1}\circ w^{k+1}-u^{k+1}).
		\]
		\STATE $k\leftarrow k+1.$
		\ENDWHILE
	\end{algorithmic}  
\end{algorithm}

\subsection{Convergence analysis of BCA}

Readily one knows that the reformulated optimization problem \eqref{eq:4}  cannot be interpreted as  a (two-block) problem (as reviewed in subsection \ref{secReviewADMM}), whose objective function is of sum of two functions without coupled variables. Due to the bilinear constraint and coupled term $v_i\ln w_i$ in the objective function, it does not also belong to the problems considered either  for convex optimization problem \cite{lin2015global,chen2016direct,deng2017parallel} with objective functions of sum of no less than three functions  or nonconvex optimization problem  \cite{Hajinezhad2018,doi:10.1137/18M1188446} with bilinear constraint. The current algorithm  introduces similar bilinear constraint as for the blind ptychography problem in \cite{doi:10.1137/18M1188446}. However, the proof technique for \cite{doi:10.1137/18M1188446} cannot directly apply to the current BCA, since the linear relation between the iterative multipliers and the auxiliary variable does not hold for BCA, and more specifically, their relation is bilinear  (See Lemma \ref{le:1}). Moreover, the coercivity of the objective function is not trivial. 
Therefore, one has to develop a new technique for convergence guarantee. 

To guarantee the sufficient decrease and boundedness of the iterative sequence, we make the following assumption. Although limited by current analysis technique we cannot remove it,  it can be verified numerically (See Fig. \ref{fig7} in the numerical part of this paper).
\begin{assumption}
\label{asum}
The iterative sequence $\{w^k\}$ generated by BCA algorithm has a uniformly positive lower bound, i.e., $w_i^k\ge c>0, \forall i$, where c is a positive constant which is independent to $k$.
\end{assumption}
\vskip .1in

\begin{lemma}
\label{le:2}
Let $T(x)=\frac{1}{2}\|Ax-b\|^{2}+M(x)$, with convex function $M$. Letting $x^*$ be a stationary point of $T(x)$ (also a global minimizer), i.e. $0\in\partial T(x^*)$, {where $\partial T(x)$ denotes the subdifferential of $T(x)$ in the convex analysis sense,} then we have
\begin{displaymath}
T(x)-T(x^*)\ge\|A(x-x^*)\|^{2}.
\end{displaymath}
\end{lemma}
\begin{proof}
Let $H(x)=\frac{1}{2}\|Ax-b\|^{2}$. Since $x^*$ is a stationary point, i.e. 
$
0 \in \nabla H(x^*)+\partial M(x^*),
$
readily one has
\begin{displaymath}
M(x)-M(x^*)\ge\langle -\nabla H(x^*),x-x^*\rangle\ \ \ \forall ~ x.
\end{displaymath}
Then we have 
\begin{displaymath}
\begin{aligned}
T(x)-T(x^*)\ge H(x)-H(x^*)-\langle \nabla H(x^*),x-x^*\rangle=\frac{1}{2}\|A(x-x^*)\|^2,
\end{aligned}
\end{displaymath}
that immediately concludes this lemma.
\end{proof}
\vskip .1in

In the following, within the framework in \cite{wang2015global,doi:10.1137/18M1188446,HongLuo,Lou2018,Mei2018,Hajinezhad2018} developed for the analysis of ADMM for nonconvex nonsmooth optimization problem, we will first prove that the iterative sequence satisfies the sufficient decrease condition. Then, the relative error condition for the iterative sequence will be derived. Finally one can derive the subsequence convergence of the proposed BCA. 

\begin{lemma}
\label{le:3}
Letting $(u^k,v^k,w^k,\Lambda^k)$ be the sequence generated by BCA in Algorithm 1, and $\alpha>\frac{\sqrt{2}\lambda_2}{c^2\epsilon} $, then under Assumption \ref{asum} we have
\begin{equation}
\label{eq:15}
\begin{aligned}
&\ \ \ \ L_{\alpha}(u^k,v^k,w^k,\Lambda^k)-L_{\alpha}(u^{k+1},v^{k+1},w^{k+1},\Lambda^{k+1})\ge\frac{\alpha}{2}\|u^{k+1}-u^{k}\|^2\\
&+\frac{\lambda_1}{2}\|v^{k+1}-v^{k}\|^2+\frac{\alpha}{2}\|w^k\circ(v^{k+1}-v^{k})\|^2+C_1\|v^{k+1}\circ(w^{k+1}-w^{k})\|^2,
\end{aligned}
\end{equation}
where $C_1$ is a positive constant which is independent to $k$.
\end{lemma}
\vskip .1in

\begin{proof}
For $u$-subproblem, by Lemma \ref{le:2}, one readily has
\begin{equation}
\label{eq:16}
\begin{aligned}
L_{\alpha}(u^k,v^k,w^k,\Lambda^k)-L_{\alpha}(u^{k+1},v^{k},w^{k},\Lambda^{k})\ge\frac{\alpha}{2}\|u^{k+1}-u^k\|^2.
\end{aligned}
\end{equation}

Similarly using Lemma \ref{le:2}, for $v$-subproblem, one can obtain
\begin{equation}
\label{eq:17}
\begin{aligned}
 &L_{\alpha}(u^{k+1},v^k,w^k,\Lambda^k)-L_{\alpha}(u^{k+1},v^{k+1},w^{k},\Lambda^{k})\\
\ge&\frac{\lambda_1}{2}\|v^{k+1}-v^k\|^2+\frac{\alpha}{2}\|w^k\circ(v^{k+1}-v^k)\|^2.
\end{aligned}
\end{equation}

For $w$-subproblem, one gets
\begin{equation}
\label{eq:18}
\begin{aligned}
&\ \ \ \ L_{\alpha}(u^{k+1},v^{k+1},w^k,\Lambda^k)-L_{\alpha}(u^{k+1},v^{k+1},w^{k+1},\Lambda^{k})\\
&\ge\frac{\alpha}{2}\|v^{k+1}\circ(w^{k+1}-w^k)\|^2.
\end{aligned}
\end{equation}

By \eqref{eq:9}, Lemma \ref{le:1} and Assumption \ref{asum} one has
\begin{equation}
\label{eq:19}
\begin{aligned}
&\ \ \ \ L_{\alpha}(u^{k+1},v^{k+1},w^{k+1},\Lambda^k)-L_{\alpha}(u^{k+1},v^{k+1},w^{k+1},\Lambda^{k+1})\\
&=-\frac{1}{\alpha}\left\|\Lambda^{k+1}-\Lambda^k\right\|^2=-\frac{1}{\alpha}\left\|\myatop{\lambda_2\mathbf{1}}{w^{k+1}}-\myatop{\lambda_2\mathbf{1}}{w^{k}}\right\|^2\\
&=-\frac{\lambda_2^2}{\alpha}\left\|\myatop{v^{k+1}\circ(w^{k+1}-w^{k})}{v^{k+1}\circ w^{k+1}\circ w^k}\right\|^2\\
&\ge-\frac{\lambda_2^2}{\alpha c^4\epsilon^2}\left\|\myatop{v^{k+1}\circ(w^{k+1}-w^{k})}{w^{k+1}\circ w^k}\right\|^2.
\end{aligned}
\end{equation}
Since $\alpha>\frac{\sqrt{2}\lambda_2}{c^2\epsilon} $, further by \eqref{eq:16}-\eqref{eq:19}, one can conclude to this lemma.
\end{proof}
\vskip .1in
\begin{lemma}
\label{le:4}
Denote $\mathcal G:\Omega\rightarrow \mathbb R$ by
\[
\mathcal{G}(u,v,w)=\frac{\lambda_1}{2}\left\|f-v\right\|^2+\lambda_2\left\langle(\mathbf{1}-\myatop{\mathbf{1}}{w})\circ u-v\circ\ln w,\mathbf{1}\right\rangle+\frac{\alpha}{2}\left\|v\circ w-u\right\|^2,\]
with $\alpha>\lambda_2\left(\frac{1}{c}-1\right)^2$, 
where $\Omega:=\{(u,v,w)\mid  v_i\ge\epsilon>0, w_i\ge c>0\ \forall ~ i; u,v,w\in \mathbb{R}^n \}.$ 
If $\|(u,v,w)\|_{\Omega}:=\max\{\|u\|_{\infty},\|v\|_{\infty},\|w\|_{\infty}\}\to+\infty$, then we have $\mathcal{G}(u,v,w)\to+\infty$.
\end{lemma}

\begin{proof}
For all $(u, v, w)\in \Omega$, one readily has 
\begin{equation}
\label{eq:20}
\begin{aligned}
&\ \ \ \ \mathcal{G}(u,v,w)\geq \frac{\lambda_1}{2}\|f-v\|^2+\lambda_2\left\langle(\mathbf{1}-\myatop{\mathbf{1}}{w})\circ u-v\circ w,\mathbf{1}\right\rangle+\frac{\alpha}{2}\|v\circ w-u\|^2\\
& = \frac{\lambda_1}{2}\|f-v\|^2+\lambda_2\left\langle(\mathbf{1}-\myatop{\mathbf{1}}{w})\circ (u-v\circ w)-v,\mathbf{1}\right\rangle+\frac{\alpha}{2}\|v\circ w-u\|^2\\
&\ge \frac{\lambda_1}{2}\|f-v\|^2-\lambda_2\left\langle\vert \mathbf{1}-\myatop{\mathbf{1}}{w}\vert \circ \vert u-v\circ w\vert +v,\mathbf{1}\right\rangle+\frac{\alpha}{2}\|v\circ w-u\|^2\\
&= \frac{\lambda_1}{2}\|f-v\|^2-\lambda_2\langle v,\mathbf{1}\rangle+\frac{\lambda_2}{2}\left\|\vert \mathbf{1}-\myatop{\mathbf{1}}{w}\vert \circ \vert u-v\circ w\vert -\mathbf{1}\right\|^2\\
&\qquad \quad -\frac{\lambda_2}{2}\left\|\vert \mathbf{1}-\myatop{\mathbf{1}}{w}\vert \circ \vert u-v\circ w\vert \right\|^2-\frac{\lambda_2}{2}\|\mathbf{1}\|^2+\frac{\alpha}{2}\|v\circ w-u\|^2\\
&\stackrel{\mathrm{Assumption~} \ref{asum}}{ \geq} \frac{\lambda_1}{2}\|f-v\|^2-\lambda_2\langle v,\mathbf{1}\rangle+\frac{\lambda_2}{2}\left\|\vert \mathbf{1}-\myatop{\mathbf{1}}{w}\vert \circ \vert u-v\circ w\vert -\mathbf{1}\right\|^2\\
&\qquad\quad -\frac{\lambda_2}{2}\left(\frac{1}{c}-1\right)^2\| u-v\circ w \|^2-\frac{\lambda_2 n}{2}+\frac{\alpha}{2}\|v\circ w-u\|^2\\
&=\frac{\lambda_1}{2}\|f-v\|^2-\lambda_2\langle v,\mathbf{1}\rangle+\frac{\lambda_2}{2}\left\|\vert \mathbf{1}-\myatop{\mathbf{1}}{w}\vert \circ \vert u-v\circ w\vert -\mathbf{1}\right\|^2\\
&\qquad \quad +\left[\frac{\alpha}{2}-\frac{\lambda_2}{2}\left(\frac{1}{c}-1\right)^2\right]\left\|v\circ w-u\right\|^2-\frac{\lambda_2 n}{2},
\end{aligned}
\end{equation}
where the first inequality is derived by  $-\ln w_i\ge -w_i$ if $w_i>0$. 

In the following part, we consider the following two cases for $\|(u,v,w)\|_\Omega\rightarrow +\infty$.\\
\noindent\textbf{Case 1:} $\|v\|_{\infty}\to +\infty$ or $\|v\circ w- u\|_{\infty}\to+\infty$. Since $\alpha>\lambda_2(\frac{1}{c}-1)^2$, one can readily get that $\mathcal{G}(u,v,w)\to+\infty$. 
\newline
\textbf{Case 2:} There exists two constants $C_2, C_3>0$, such that $\|v\|_{\infty}\le C_2<+\infty$, $\|u\|_{\infty}\to+\infty$, $\|w\|_{\infty}\to+\infty$, and $\|v\circ w- u\|_{\infty}\le C_3<+\infty$. Then we have
\begin{equation}
\label{eq:21}
\begin{aligned}
\mathcal{G}(u,v,w)=\sum\limits_i[\frac{\lambda_1}{2}(f_i-v_i)^2+\lambda_2(1-\frac{1}{w_i})u_i-v_i\ln w_i+\frac{\alpha}{2}(v_iw_i-u_i)^2 ].
\end{aligned}
\end{equation}
There must exist some $i$ where $u_i\to+\infty$, $w_i\to+\infty$, and $\vert v_i w_i-u_i\vert\le C_3$. Thus, we have
\begin{equation*}
\begin{aligned}
&\epsilon w_i-u_i\le\vert v_i w_i- u_i\vert\le C_3.
\end{aligned}
\end{equation*}
Then we get the lower bound estimate of $u_i$ as 
\begin{equation}
\label{eq:22}
\begin{aligned}
u_i\ge \epsilon w_i+C_3.
\end{aligned}
\end{equation}
Therefore,
\begin{equation}
\label{eq:23}
\begin{aligned}
&\ \ \ \ \frac{\lambda_1}{2}(f_i-v_i)^2+\lambda_2(1-\frac{1}{w_i})u_i-v_i\ln w_i+\frac{\alpha}{2}(v_iw_i-u_i)^2\\
&\ge \lambda_2(1-\frac{1}{w_i})u_i-v_i\ln w_i\\
&\ge \lambda_2(1-\frac{1}{w_i})(\epsilon w_i+C_3)-C_2\ln w_i\\
&=\lambda_2\epsilon w_i-\lambda_2\epsilon+\lambda_2C_3-\frac{\lambda_2 C_3}{w_i}.
\end{aligned}
\end{equation}
Since $\lim\limits_{w\to+\infty}\frac{\lambda_2\epsilon w-\lambda_2\epsilon+\lambda_2C_3-\frac{\lambda_2 C_3}{w}}{w}=\lambda_2\epsilon>0$, we can readily get $(\lambda_2\epsilon w-\lambda_2\epsilon+\lambda_2C_3-\frac{\lambda_2 C_3}{w})\to+\infty$ as $w\to+\infty$. Thus, we can derive that 
\[
\frac{\lambda_1}{2}(f_i-v_i)^2+\lambda_2(1-\frac{1}{w_i})u_i-v_i\ln w_i+\frac{\alpha}{2}(v_iw_i-u_i)^2\to+\infty.
\]

If the variables $u_j, v_j, w_j$ do not  satisfy the  above two cases, they cannot tend to the infinity. Hence, in summary, we can conclude that $\mathcal{G}(u,v,w)\to+\infty$ as $\|(u, v, w)\|_\Omega\to +\infty$.
\end{proof}
\vskip .1in

\begin{thm}
Letting $\alpha>\max(\frac{\sqrt{2}\lambda_2}{c^2\epsilon},\lambda_2(\frac{1}{c}-1)^2) $, under Assumption \ref{asum}, we have
\begin{itemize}
\item [(1)]The sequence $(u^k,v^k,w^k,\Lambda^k)$ generated by proposed BCA is bounded and has at least one limit point.
\item [(2)]The successive errors $u^{k+1}-u^k\to0$, $v^{k+1}-v^k\to0$, $w^{k+1}-w^k\to0$, and $\Lambda^{k+1}-\Lambda^k\to0$ as $k\to +\infty$.
\item [(3)]Each limit point $(u^*,v^*,w^*,\Lambda^*)$ is a stationary point of $L_{\alpha}(u,v,w,\Lambda)$, and $(u^*,v^*)$ is a stationary point of $H(u,v)$.
\end{itemize}
\end{thm}

\begin{proof}
(1) If $\alpha>\frac{\sqrt{2}\lambda_2}{c^2\epsilon}$, by Lemma \ref{le:3}, we get
\begin{equation}
\label{eq:24}
\begin{aligned}
&\ \ \ \ L_{\alpha}(u^k,v^k,w^k,\Lambda^k)-L_{\alpha}(u^{k+1},v^{k+1},w^{k+1},\Lambda^{k+1})\\
&\ge\frac{\alpha}{2}\|u^{k+1}-u^{k}\|^2+\frac{\lambda_1}{2}\|v^{k+1}-v^{k}\|^2+\frac{\alpha}{2}\|w^k\circ(v^{k+1}-v^{k})\|^2\\
&\hskip 5.5cm +C_1\|v^{k+1}\circ(w^{k+1}-w^{k})\|^2\\
&\ge\frac{\alpha}{2}\|u^{k+1}-u^{k}\|^2+\frac{\lambda_1+\alpha c^2}{2}\|v^{k+1}-v^{k}\|^2+C_1\epsilon^2\|w^{k+1}-w^{k}\|^2.
\end{aligned}
\end{equation}

Next, we will show that $L_{\alpha}(u^k,v^k,w^k,\Lambda^k)$ is lower bounded.
Readily one knows that 
\begin{equation}
\label{eq:25}
\begin{aligned}
&\ \ \ \ L_{\alpha}(u^k,v^k,w^k,\Lambda^k)\\
&\ge \frac{\lambda_{1}}{2}\left\|f-v^{k}\right\|^{2}+\lambda_{2}\left\langle (\mathbf{1}-\myatop{\mathbf{1}}{w^k})\circ u^k- v^{k} \circ \ln w^{k}, \mathbf{1}\right\rangle+\frac{\alpha}{2}\left\|v^{k} \circ w^{k}-u^{k}\right\|^{2}\\
&=\mathcal{G}(u^k,v^k,w^k).
\end{aligned}
\end{equation}
Then, following  \eqref{eq:25} and Lemma \ref{le:4},  the sequences $\{u^k\}, \{v^k\}, \{w^k\}$ and $\{\mathcal{G}(u^k,v^k,w^k)\}$ are all bounded as well as the boundedness of $\{\Lambda^k\}$ due to Lemma \ref{le:1}. 

Due to the boundedness of $(u^k,v^k,w^k,\Lambda^k)$, there exists a convergent subsequence $(u^{k_i},v^{k_i},w^{k_i};\Lambda^{k_i})$, i.e., $(u^{k_i},v^{k_i},w^{k_i},\Lambda^{k_i})\to(u^{*},v^{*},w^{*},\Lambda^{*})$.
\par
(2)\ By \eqref{eq:25}, one readily knows that the sequence $L_{\alpha}(u^k,v^k,w^k,\Lambda^k)$ is bounded below. Therefore, further by summing up \eqref{eq:24} from $k=1$ to $\infty$ implies that 
\[\sum_{k=1}^{\infty}\left\|u^{k+1}-u^{k}\right\|^{2}+\left\|v^{k+1}-v^{k}\right\|^{2}+\left\|w^{k+1}-w^{k}\right\|^{2}<\infty.\]
That immediately implies that $u^{k+1}-u^{k} \rightarrow 0$, $v^{k+1}-v^{k} \rightarrow 0$, $w^{k+1}-w^{k} \rightarrow 0$. By Lemma \ref{le:1}, one can also knows  that $\Lambda^{k+1}-\Lambda^{k} \rightarrow 0$.

(3) It follows from the optimality condition of $u$-subproblem that there exists $q\in\partial\sum\limits_i\vert\nabla u_i^{k+1}\vert$ such that
\begin{displaymath}
q+\lambda_2\mathbf 1-\Lambda^k-\alpha(v^k\circ w^k-u^{k+1})=0.
\end{displaymath}
Letting $p=q+\lambda_2-\Lambda^{k+1}-\alpha(v^{k+1}\circ w^{k+1}-u^{k+1})\in\partial_u L_{\alpha}(u^{k+1},v^{k+1},w^{k+1};\Lambda^{k+1})$, then we have
\begin{equation}
\label{eq:26}
\begin{aligned}
\|p\|&=\|q+\lambda_2\mathbf 1-\Lambda^{k+1}-\alpha(v^{k+1}\circ w^{k+1}-u^{k+1})\|\\
&=\|\Lambda^k-\Lambda^{k+1}+\alpha(v^k\circ w^k-v^{k+1}\circ w^{k+1})\|\\
&\le \|\Lambda^{k+1}-\Lambda^k\|+\alpha\|w^k\circ(v^{k+1}-v^{k})\|+\alpha\|v^{k+1}\circ(w^{k+1}-w^k)\|.
\end{aligned}
\end{equation}

The optimality condition of $v$-subproblem implies that there exists $q_1\in\partial \chi_{\mathcal{V}}(v^{k+1})$ such that
\begin{displaymath}
q_1+\lambda_1 v^{k+1}+\alpha(w^k)^2\circ v^{k+1}-\lambda_1 f-\lambda_2\ln w^k-\lambda_2+w^k\circ \Lambda^k-\alpha w^k\circ u^{k+1}  =0.
\end{displaymath}
Letting $p_1=q_1+\lambda_1 v^{k+1}+\alpha(w^{k+1})^2\circ v^{k+1}-\lambda_1 f-\lambda_2\ln w^{k+1}-\lambda_2+w^{k+1}\circ \Lambda^{k+1}-\alpha w^{k+1}\circ u^{k+1}\in\partial_u L_{\alpha}(u^{k+1},v^{k+1},w^{k+1};\Lambda^{k+1})$, then we have
\begin{equation}
\label{eq:27}
\begin{aligned}
&\qquad \|p_1\|\\
&=\|q_1+\lambda_1 v^{k+1}+\alpha(w^{k+1})^2\circ v^{k+1}-\lambda_1 f-\lambda_2\ln w^{k+1}-\lambda_2\\
&\hskip 6cm +w^{k+1}\circ \Lambda^{k+1}-\alpha w^{k+1}\circ u^{k+1}\|\\
&=\|\alpha v^{k+1}\circ[(w^{k+1})^2-(w^k)^2]-\lambda_2(\ln w^{k+1}-\ln w^k)+\Lambda^{k+1}\circ w^{k+1}\\
&\hskip 5.85cm-\Lambda^k\circ w^k-\alpha u^{k+1}(w^{k+1}-w^k)\|\\
&=\|\alpha v^{k+1}\circ(w^{k+1}+w^k)\circ(w^{k+1}-w^k)-\lambda_2(\ln w^{k+1}-\ln w^k)\\
&\hskip 7.1cm -\alpha u^{k+1}\circ(w^{k+1}-w^k)\|\\
&=\|(\Lambda^{k+1}-\Lambda^k)\circ(w^{k+1}-w^k)+\alpha v^{k+1}\circ w^k\circ(w^{k+1}-w^k)\\
&\hskip 7.4cm-\lambda_2(\ln w^{k+1}-\ln w^k)\|\\
&\le \|(\Lambda^{k+1}-\Lambda^k)\|\|(w^{k+1}-w^k)\|+\alpha\|v^{k+1}\circ w^k\circ(w^{k+1}-w^k)\|\\
&\hskip 7.45cm+\lambda_2\|\ln w^{k+1}-\ln w^k)\|.
\end{aligned}
\end{equation}
By the optimality condition of $w$-subproblem and \eqref{eq:9}, we have 
\begin{equation}
\label{eq:28}
\begin{aligned}
&\ \ \ \ \|\nabla_{w}L_{\alpha}\left(u^{k+1}, v^{k+1}, w^{k+1}; \Lambda^{k+1}\right)\|\\
&=\|\alpha(v^{k+1})^2\circ(w^{k+1})^2+(\Lambda^{k+1}\circ v^{k+1}-\alpha v^{k+1}\circ u^{k+1})\circ w^{k+1}-\lambda_2 v^{k+1}\|\\
&=\|v^{k+1}\circ w^{k+1}\circ(\Lambda^{k+1}-\Lambda^k)\|,
\end{aligned}
\end{equation}
and
\begin{equation}
\label{eq:29}
\begin{aligned}
&\ \ \ \ \|\nabla_{\Lambda}L_{\alpha}\left(u^{k+1}, v^{k+1}, w^{k+1}; \Lambda^{k+1}\right)\|\\
&=\|v^{k+1}\circ w^{k+1}-u^{k+1}\|=\frac{1}{\alpha}\|\Lambda^{k+1}-\Lambda^k\|.
\end{aligned}
\end{equation}
Finally, \eqref{eq:26}-\eqref{eq:29} and Item (1) in this theorem suggest that $(u^{*},v^{*},w^{*};\Lambda^{*})$ is a stationary point of $L_{\alpha}(u,v,w,\Lambda)$. Since $(u^{*},v^{*},w^{*};\Lambda^{*})$ is a stationary point, we have $u^*=v^* w^*$ from \eqref{eq:29}, then \eqref{eq:26} and \eqref{eq:27} imply that $0\in\partial_u H(u^*,v^*)$ and $0\in\partial_v H(u^*,v^*)$, i.e., $(u^*, v^*)$ is a stationary point of $H(u,v)$.
\end{proof}

{
We remark that in order to prove the theoretical convergence of the proposed BCA algorithm, we assume that $\epsilon$ is a positive constant.  Simulation results reported in the experimental section of this paper will not be affected if $\epsilon$ is selected appropriately. As for the case $\epsilon=0$, we will investigate the theoretical convergence in the future. 
%We also remark that if we replace the TV term with Huber function \cite{10.2307/2238020} to modify TV-IC model, and further update the sequence by switching the orders as  $v^{k+1}, w^{k+1}, u^{k+1}$, it is not difficult to  prove the sufficient decrease condition. We can also readily prove the convergence of this algorithm by following the same framework as above regardless of Assumption \ref{asum}. 
}

\subsection{BCA$_f$}
The proposed BCA algorithm has a subproblem in {\color{red}\eqref{eq:10}} w.r.t. total variation minimization problem, which requires inner loop. To get a fully splitting scheme, we propose the following BCA$_f$ algorithm. 

We introduce one more auxiliary variable $p$ satisfying the  constraint $p_i=\nabla u_i$ ($p\in\mathbb R^{n,2}$ with $p_i\in\mathbb R^2$ as its $i^{th}$ row), in additional to the constraint $u_i=v_i w_i$ in proposed BCA, and then rewrite \eqref{eq:2} as the following equivalent constrained optimization problem:
\begin{equation}
\label{eq:30}
\begin{aligned}
&\min\limits_{(u,v)}\left\{ \frac{\lambda_{1}}{2}\sum\limits_{i}(f_{i}-v_{i})^{2}+\sum\limits_{i}| p_{i}|+\chi_{\cV}(v)+\lambda_{2}\sum\limits_{i}(u_i-v_i\ln w_i-v_i)\right\},\\
&s.t.\ \ \  u_i=v_i w_i, \ \ \  \ p_i=\nabla u_{i}.
\end{aligned}
\end{equation}
Then we can easily get the augmented Lagrangian of the above constrained optimization problem by introducing the multipliers $\Lambda_w$ and $\Lambda_p$:
\begin{equation}
\label{eq:31}
\begin{aligned}
\!\!\!\!&L_{\alpha_w,\alpha_p}(u,v,w,p,\Lambda_w,\Lambda_p)= \frac{\lambda_{1}}{2}\sum\limits_{i}(f_{i}-v_{i})^{2}+\sum\limits_{i}| p_{i}|\\
&+\lambda_{2}\sum\limits_{i}(u_i-v_i\ln w_i-v_i)+\langle \Lambda_{w},v\circ w-u\rangle+\langle \Lambda_{p},p-\nabla u\rangle\\
&+\frac{\alpha_w}{2}\| v\circ w-u\|^2+\frac{\alpha_p}{2}\| p-\nabla u\|^2+\chi_{\cV}(v),
\end{aligned}
\end{equation}
where  $\alpha_w>0$ and $\alpha_p>0$ are the penalization parameters.

Given the previous iterative solution $(u^k,v^k,w^k,p^k,\Lambda_w^k,\Lambda_p^k)$, the ADMM consists of the following iterations
\begin{subequations}
\begin{numcases}{}
u^{k+1}=\argmin_{u}L_{\alpha_w,\alpha_p}(u,v^k,w^k,p^k,\Lambda_w^k,\Lambda_p^k),\label{eq:32a}\\
v^{k+1}=\argmin_v L_{\alpha_w,\alpha_p}(u^{k+1},v,w^k,p^k,\Lambda_w^k,\Lambda_p^k),\label{eq:32b}\\
w^{k+1}=\argmin_w L_{\alpha_w,\alpha_p}(u^{k+1},v^{k+1},w,p^k,\Lambda_w^k,\Lambda_p^k),\label{eq:32c}\\
p^{k+1}=\argmin_p L_{\alpha_w,\alpha_p}(u^{k+1},v^{k+1},w^{k+1},p,\Lambda_w^k,\Lambda_p^k),\label{eq:32d}\\
\Lambda_w^{k+1}=\Lambda_w^k+\alpha_w(v^{k+1}\circ w^{k+1}-u^{k+1}),\label{eq:32e}\\
\Lambda_p^{k+1}=\Lambda_p^k+\alpha_p(p^{k+1}-\nabla u^{k+1}).\label{eq:32f}
\end{numcases}
\end{subequations}

We will show how to solve the subproblems w.r.t. $(u,v,w,p)$ in the rest of this part. 
We consider the $u$-subproblem below:
\begin{displaymath}
\begin{aligned}
u^{k+1}=\arg\min_{u}\left\{\lambda_2\sum\limits_i u_i+\frac{\alpha_{w}}{2}\left\| v^k\circ w^k+\frac{\Lambda_w^k}{\alpha_w}-u\right\|^{2}+\frac{\alpha_{p}}{2}\left\| p^{k}+\frac{\Lambda_p^k}{\alpha_p}-\nabla u\right\|^{2}\right\}.
\end{aligned}
\end{displaymath}
The first-order optimality condition of this subproblem is directly given below:
\begin{equation}
\label{eq:32}
\alpha_w u-\alpha_p\triangle u=-\lambda_2+\Lambda_w^k-\nabla\cdot\Lambda_p^k+\alpha_w v^k w^k-\alpha_p\nabla\cdot p^k,
\end{equation}
where $\nabla\cdot$ denotes the divergence operator (conjugate operator of negative gradient $-\nabla$). 
We can readily solve the above equations  by using conjugate gradient (CG) method or fast Fourier transform. 

For  $v, w$-subproblems (same to BCA, but with different notations), one can readily obtain that
\begin{equation}
\label{eq:33}
v^{k+1}=\max\left(\epsilon\mathbf 1,\tilde{v}^{k+1}\right),
\end{equation}
where
\begin{displaymath}
\tilde{v}^{k+1}=\myatop{\mathbf{1}}{\lambda_1\mathbf{1}+\alpha_w(w^k)^2}\circ\left(\lambda_1 f+\lambda_2\ln w^k+\lambda_2-w^k\circ\Lambda_w^k+\alpha_w w^k\circ u^{k+1}\right).
\end{displaymath}
and 
\begin{equation}
\label{eq:34}
w_i^{k+1}=\frac{1}{2v_i^{k+1}}\left[(u_i^{k+1}-\frac{\Lambda_{w(i)}^k}{\alpha_w})+\sqrt{(u_i^{k+1}-\frac{\Lambda_{w(i)}^k}{\alpha_w})^2+\frac{4\lambda_2 v_i^{k+1}}{\alpha_w}}\ \right].
\end{equation}

For the $p$-subproblem, one has
\begin{displaymath}
p^{k+1}=\argmin_{p}\left\{\sum\limits_{i}| p_{i}|+\frac{\alpha_{p}}{2}\left\| p+\tfrac{\Lambda_p^k}{\alpha_p}-\nabla u^{k+1}\right\|^{2}\right\}.
\end{displaymath}
The solution is exactly the soft thresholding of $\nabla u^{k+1}-\frac{\Lambda_{p}^{k}}{\alpha_{p}}$:
\begin{equation}
\label{eq:35}
p^{k+1}=\mathrm{Thresh}_{\frac{1}{\alpha_{p}}}(\nabla u^{k+1}-\tfrac{\Lambda_{p}^{k}}{\alpha_{p}}),
\end{equation}
where  $\mathrm{Thresh}_{\eta}(p):=\max\{0,\vert p\vert - \eta\}\circ\mathrm{sign}(p),$ and {$\mathrm{sign}(p):=\big(\myatop{p^{(1)}}{\vert p\vert},\myatop{p^{(2)}}{\vert p\vert}\big)$}, $\vert p\vert=\sqrt{|p^{(1)}|^2+|p^{(2)}|^2}$. Note that all the operations here are element-wise.

Finally, the overall algorithm summarizing the above analysis is given below:
\begin{algorithm}[H]
	\renewcommand{\algorithmicrequire}{\textbf{Input:}}
	\renewcommand{\algorithmicensure}{\textbf{Initialization:}}
	\caption{BCA$_f$}
	\label{alg:2}
	\begin{algorithmic}[1]
		\REQUIRE {Noisy data $f$ and parameters $\lambda_1$, $\lambda_2$, $\alpha_w$, $\alpha_p$.}\\
		\ENSURE {{ $u^{0}=f, v^{0}=f, w^0=\mathbf 1, p^{0}=\mathbf 0, \Lambda_w^{0}=\mathbf 0, \Lambda_p^{0}=\mathbf 0, k=0.$}}\\
		\WHILE{Stopping criteria is not satisfied}
		\STATE Solve  $u^{k+1}$ by using CG for \eqref{eq:32}
		\STATE Solve  $v^{k+1}$ by \eqref{eq:33}
		\STATE Solve  $w^{k+1}$ by \eqref{eq:34}
		\STATE Solve  $p^{k+1}$ by \eqref{eq:35}
		\STATE 
		Update the multipliers by 
		\[
		\begin{split}
		&\Lambda_w^{k+1}=\Lambda_{w}^{k}+\alpha_w(v^{k+1}\circ w^{k+1}-u^{k+1});\\
		&\Lambda_{p}^{k+1}=\Lambda_{p}^{k}+\alpha_{p}(p^{k+1}-\bigtriangledown u^{k+1}).
		\end{split}
		\]
		\STATE $k\leftarrow k+1.$
		\ENDWHILE
	\end{algorithmic}  
\end{algorithm}

Here we remark that the convergence study of BCA$_f$ seems more difficult. If directly following the technique for BCA in the last subsection, since the subproblem for $p$ is non-differentiable, the successive errors of $\Lambda_p$ cannot be controlled by  the successive errors of $p$ such that it seems impossible to guarantee the sufficient decrease of the whole iterative sequences. In the future, we will either develop more advanced technique to control this error, or investigate other regularization terms with Lipschitz continuous gradient.

\section{Numerical Experiments}
\label{sec:experiments}

Since Poisson noise is data-dependent, the noise level of the observed images depends on the pixel value, and therefore we introduce a scale factor  $\eta\in(0,\infty)$ to control the scale of the image (simulating different number of photons detector received), which is inversely proportional to the amount of noise added to the data, i.e. $v_{i}\sim \mathrm{Poisson}(\eta u)/\eta$. Meanwhile, we add Gaussian noise with different variances $\sigma^2$.

The Signal-to-Noise Ratio (SNR) in dB is used to measure the quality of the recovery result,  defined as:
$\mathrm{SNR}(u,u_{g})=-10\log_{10}\tfrac{\sum\limits_{i}\vert u_i-{(u_{g}})_i\vert^2}{\sum\limits_{i}\vert u_i\vert^{2}},$
where $u_{g}$ is the ground-truth image (See Fig. \ref{fig1}) and $u$ is the reconstructed image. The structural similarity (SSIM) index  \cite{1284395} is also provided to measure the quality  of restored results  
%{\color{red}
%(When two images are nearly identical, their SSIM is close to 1).
(Smaller values mean the better quality).

\begin{figure}[htb]
 \centering 
 \subfloat[]{\includegraphics[width=0.3\columnwidth]{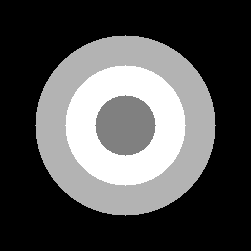}}\
 \subfloat[]{\includegraphics[width=0.3\columnwidth]{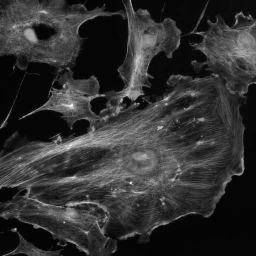}}\
 \subfloat[]{\includegraphics[width=0.3\columnwidth]{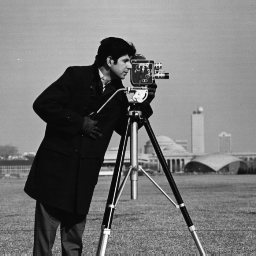}}
 \caption{(a) Circles; (b) Fluorescent Cells; (c) Cameraman}
 \label{fig1}
\end{figure}

We set the stopping criterion as the successive error $\mathrm{SE}:=\tfrac{\Vert u_{k+1}-u_{k}\Vert}{\Vert u_{k}\Vert}\leq \xi$ or  the iteration  reaches 1000, where $\xi$ is a desired tolerance.
{To evaluate the performance of the proposed algorithms, we compare them with other operator-splitting algorithms, including the $L^2$ data fitting (TV+$L^2$) method \cite{wu2010augmented}, KL-divergence (TV+KL) method \cite{wu2011augmented}, Shifted-Poisson (TV+SP) method \cite{article3}, the combination of $L^2$ data fitting and KL-divergence  (TV+KL+$L^2$) method \cite{article14} and the primal-dual (TV+PD) method for \eqref{eq:2} following \cite{article10}. The last three algorithms were specially designed for the MPG noise; especially, TV+SP is simple to implement, and able to  produce comparable results reported in \cite{article11,article8}. 
%Note that except for the TV+PD method, the other methods are all computed by ADMM-type algorithms.
}
All the parameters for the compared algorithms are tuned heuristically to gain optimal image quality. For fair comparison, we set $\xi=5\times 10^{-4}$ for all compared algorithms and set initialization of the variable w.r.t. reconstructed output to noisy images.

All compared algorithms are implemented in Matlab, and performed using a Laptop with Intel Core i5 processor and 8GB RAM. 

\subsection{Performances and convergence}

We first show how to determine the optimal inner iteration number for the proposed BCA, where we employ the gradient descent method proposed by Chambolle  \cite{Chambolle:2004:ATV:964969.964985} to solve the $u$-subproblem. To find the optimal inner iteration number (More inner iterations will increase the overall computational cost), we plot the SNR curves of different inner iteration numbers ($1,2,5,10,20,100$)  in Fig. \ref{fig8}. Obviously one sees that when the inner iteration number is greater than $10$, the SNR value is almost unchanged. Hence, in the latter tests,  the inner loop number set to $10$ for the proposed BCA.
\begin{figure}[]
 \centering 
 \subfloat[]{\includegraphics[width=0.3\columnwidth]{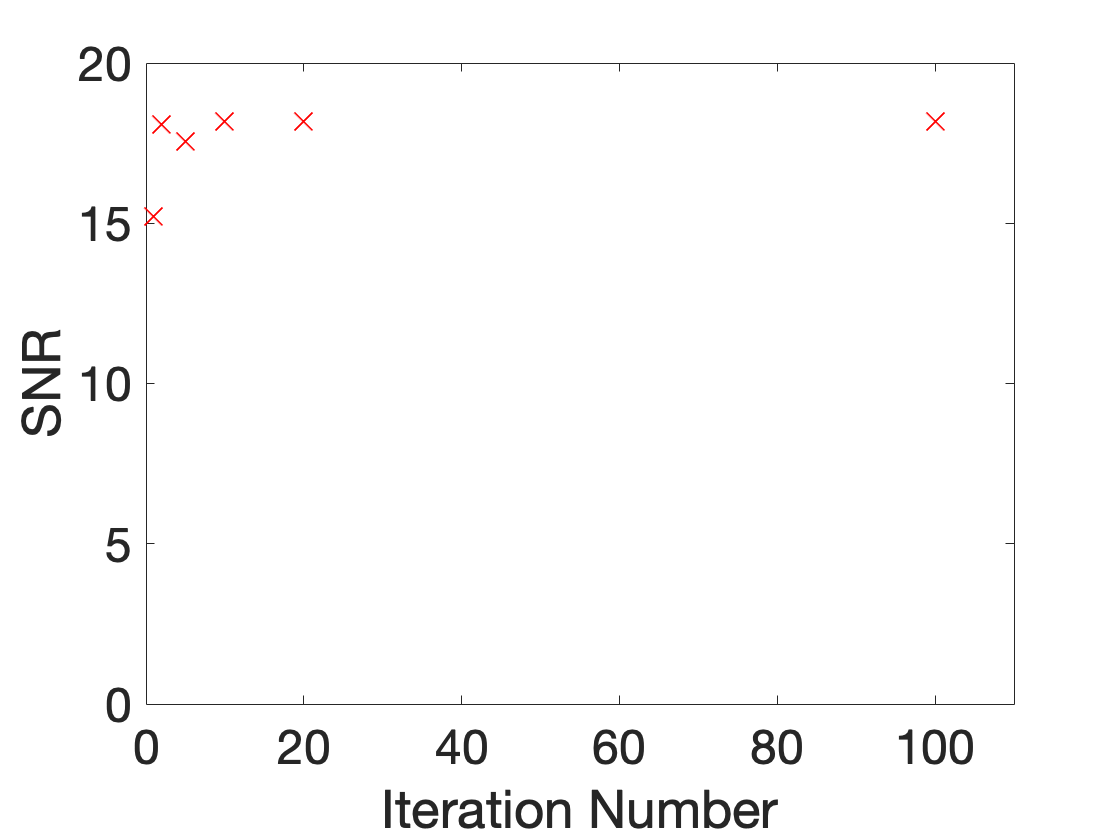}}\
 \subfloat[]{\includegraphics[width=0.3\columnwidth]{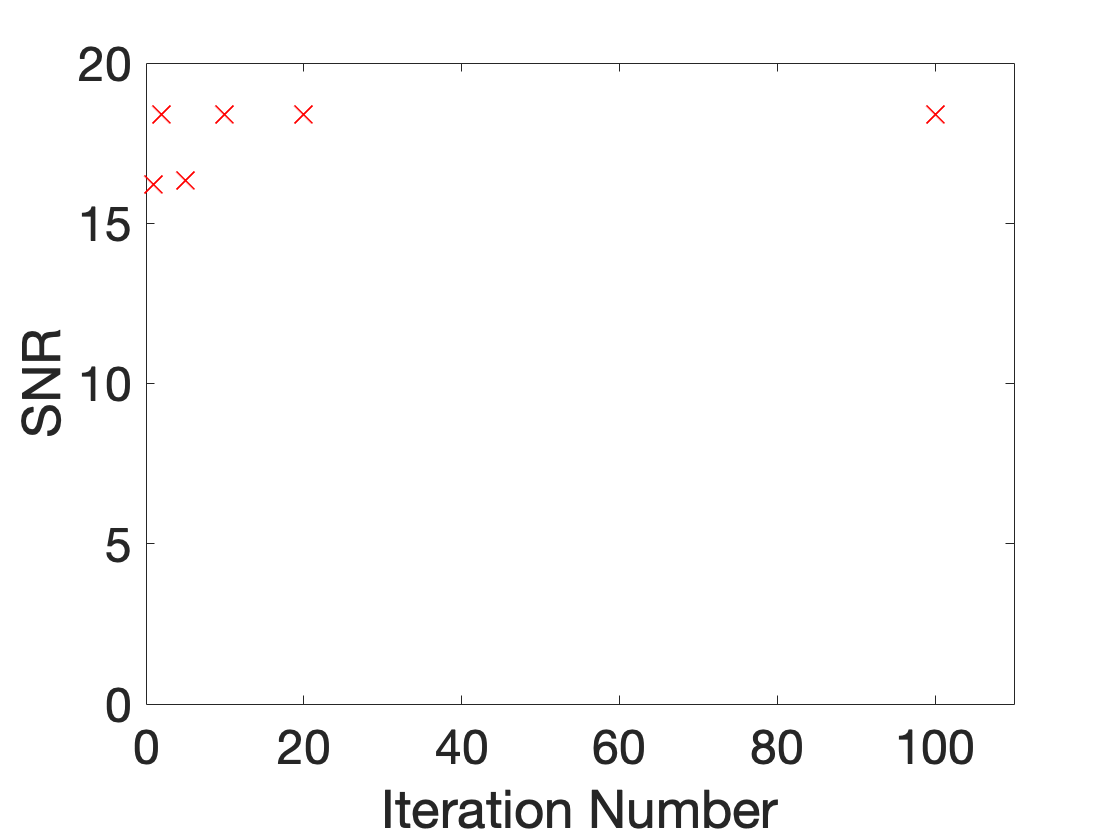}}\
 \subfloat[]{\includegraphics[width=0.3\columnwidth]{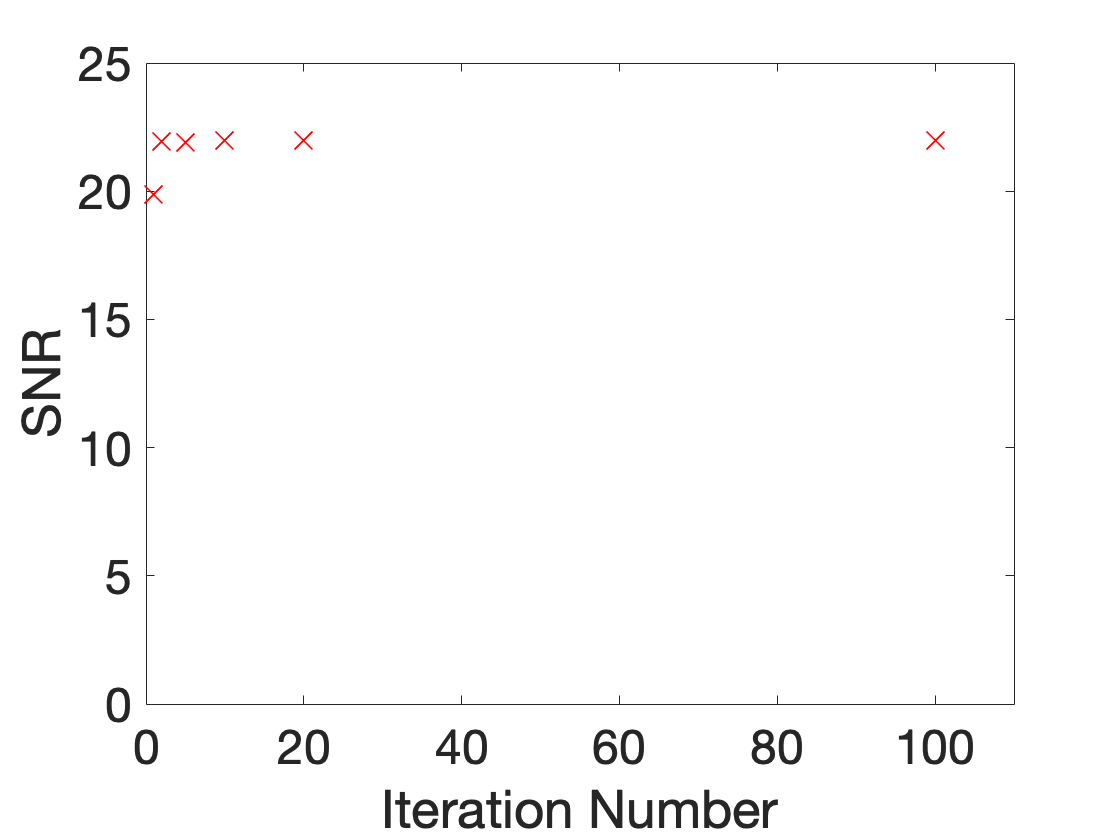}}\\
 \subfloat[]{\includegraphics[width=0.3\columnwidth]{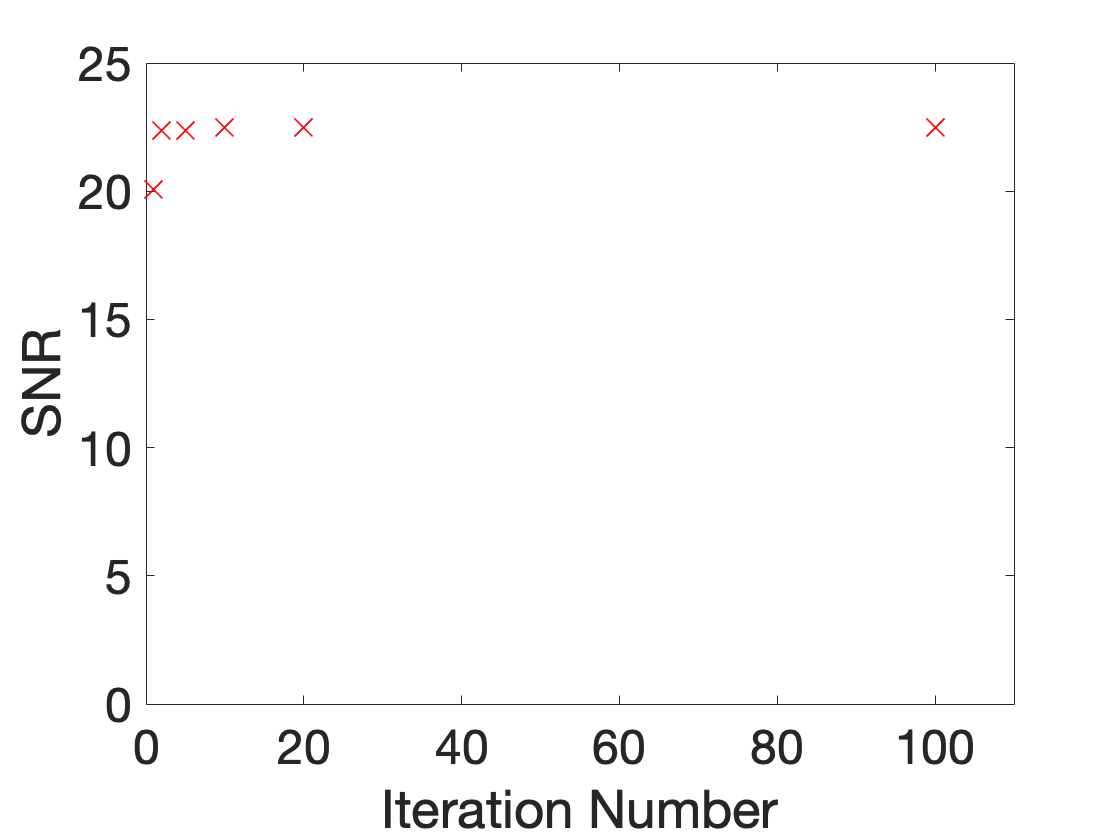}}\
 \subfloat[]{\includegraphics[width=0.3\columnwidth]{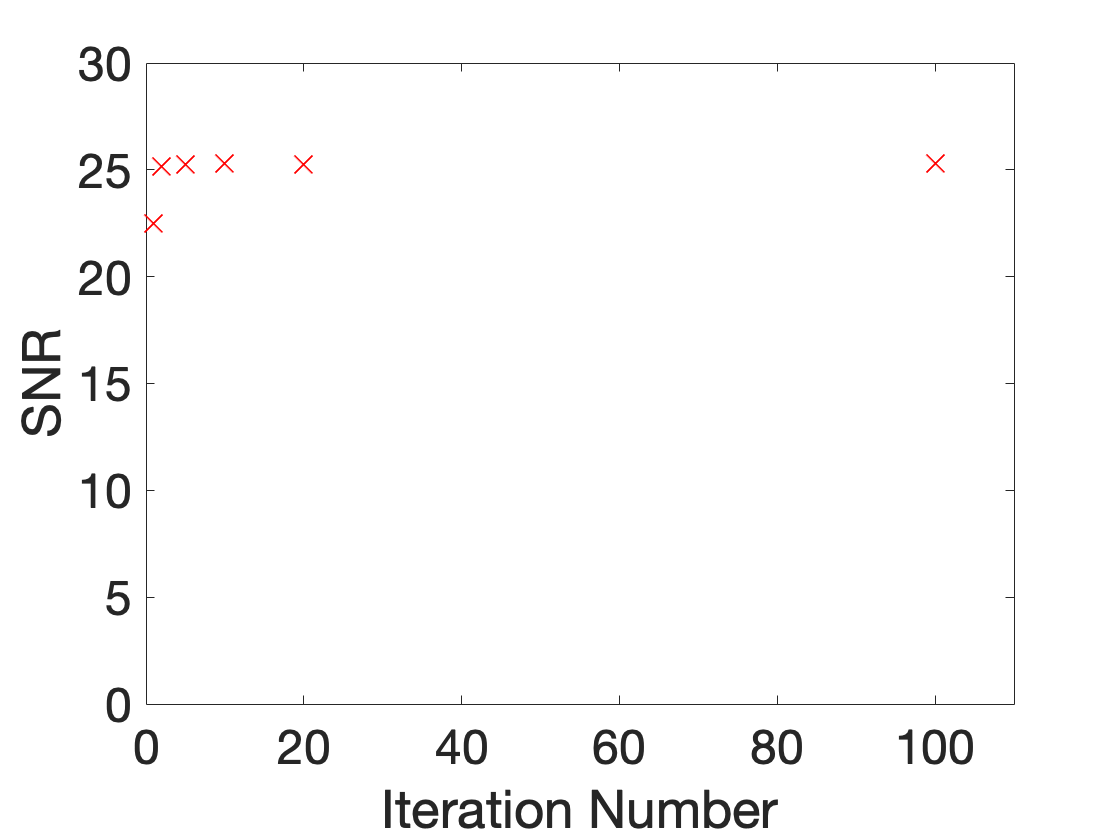}}\
 \subfloat[]{\includegraphics[width=0.3\columnwidth]{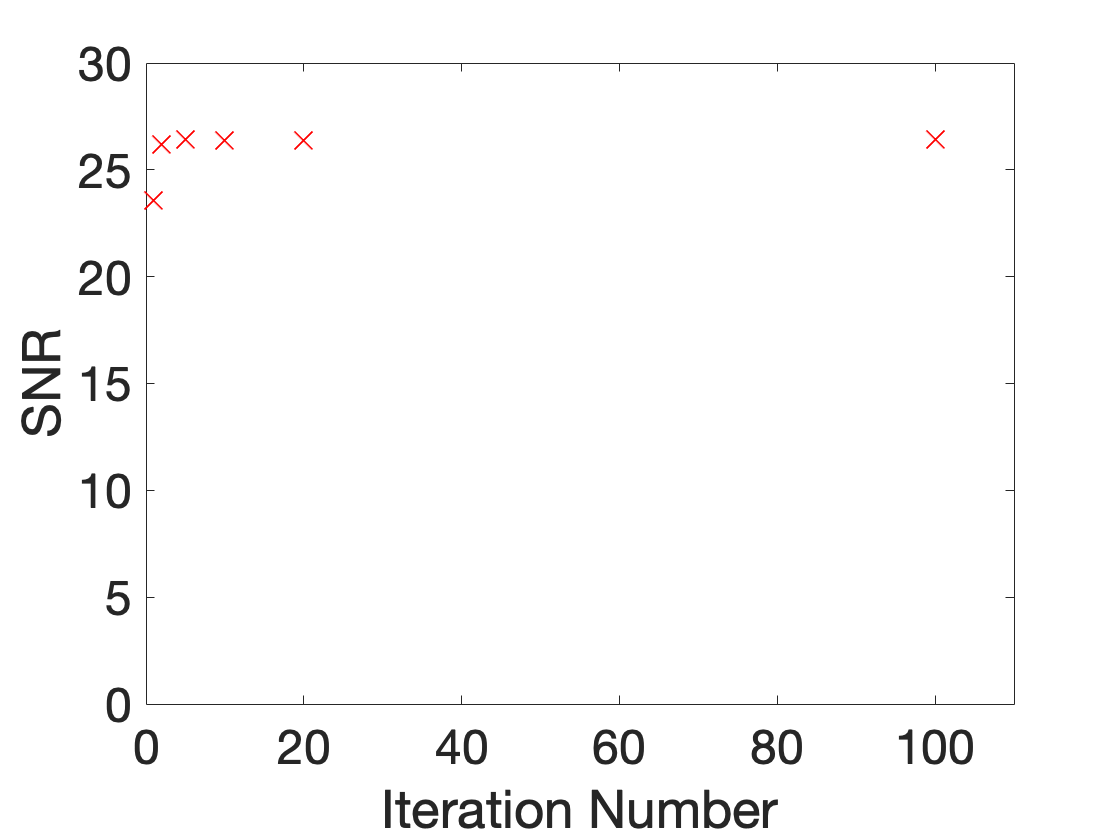}}
 \caption{SNR changes w.r.t. the number of inner iterations using gradient descent method \cite{Chambolle:2004:ATV:964969.964985}  for proposed BCA Algorithm.}
 \label{fig8}
\end{figure}

 We evaluate  the performances of our proposed methods, compared with five other methods. The recovery results with zoomed regions are put in Figs. \ref{fig2}-\ref{fig22}, 
 where  noises are generated with $\eta=4, \sigma=10^{-4}$, $\eta=16, \sigma=10^{-4}$, and $\eta=64, \sigma=10^{-1}$ for the three different images respectively. Generally speaking, one readily sees that the proposed BCA and BCA$_f$ generate better results compared with denoising methods  including TV+$L^2$, TV+KL and TV+SP. In Fig. \ref{fig2}, one can observe that the region  located at the red circles in the recovery results by the proposed BCA and BCA$_f$, especially the part below the edges, appears more flat than other compared algorithms. The recovery accuracy of recovery results by proposed algorithms with higher SNRs is also better than other compared algorithms, inferred from Fig. \ref{fig2}. In Figs. \ref{fig21} and \ref{fig22}, the recovery results by proposed algorithms look better than those by  TV+$L^2$, TV+KL and TV+SP, while look quite similar to those by TV+KL+$L^2$ and  TV+PD algorithms. Table \ref{tab0} reports the SNRs and SSIMs of recovery images for all compared algorithms with more different noisy levels ($\eta=1, 4, 16$ and $\sigma=10^{-1}, 10^{-4}$), that
 %{\color{red} (I think 'which' might be better.)}
 demonstrates that the proposed algorithms gain highest SNRs averagely. 
 
 In order to further show the advantage in term of speed for the proposed algorithms compared with TV+PD for the same model, we report the SNRs changes w.r.t. the elapsed CPU time in Fig. \ref{fig3}. 
 Readily one can see the proposed algorithm converges much faster  than TV+PD\footnote{The iteration number for subproblems solved by Newton method affect the convergence speed of TV+PD, 
 and 5 iterations are adopted heuristically to gain best speed.}. 
 Table \ref{tab1} reports the computational time of the proposed algorithms and TV+PD. It is obvious that the proposed algorithms have higher speed than TV+PD. The fully splitting algorithm BCA$_f$ computes much faster than BCA.
 We also remark that our proposed BCA and BCA$_f$ have fewer parameters (three parameters for BCA, and 
 four  parameters for BCA$_{f}$), while compared TV+PD has six parameters. 
\begin{figure}[]
\begin{center}
\begin{tabular}{cclcl}
\includegraphics[width=2.5cm]{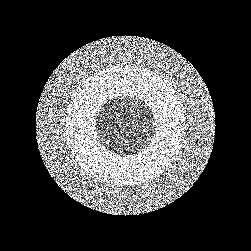} &
\begin{overpic}[width=2.5cm]{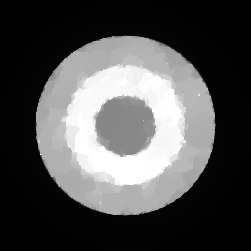}
    \put(15,64){\tikz\draw[red,thick,dashed] (0,0) rectangle (.3,.3);}
\end{overpic} &
\adjustbox{width=2.5cm,trim={.13\width} {.6\height} {0.72\width} {.25\height},clip}
{\begin{overpic}[width=2.5cm]{cirl42.png}
    \put(17,65){\tikz\draw[red,thick] (0,0) circle (.08);}
\end{overpic}
}&
\begin{overpic}[width=2.5cm]{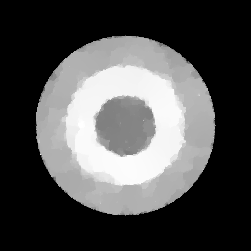}
    \put(15,64){\tikz\draw[red,thick,dashed] (0,0) rectangle (.3,.3);}
\end{overpic} &
\adjustbox{width=2.5cm,trim={.13\width} {.6\height} {0.72\width} {.25\height},clip}
{\begin{overpic}[width=2.5cm]{cirkl42.png}
    \put(17,65){\tikz\draw[red,thick] (0,0) circle (.08);}
\end{overpic}}  \\
 \makecell[tl]{\ \ \ Corruption: \\$\eta=4, \sigma=10^{-4}$ } & TV+$L^2$ : 21.64 & Zoomed & TV+KL: 21.83 & Zoomed  \\[2mm]
&
\begin{overpic}[width=2.5cm]{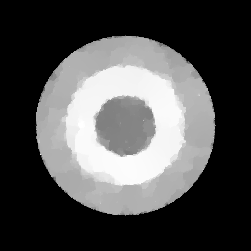}
    \put(15,64){\tikz\draw[red,thick,dashed] (0,0) rectangle (.3,.3);}
\end{overpic} &
\adjustbox{width=2.5cm,trim={.13\width} {.6\height} {0.72\width} {.25\height},clip}
{\begin{overpic}[width=2.5cm]{cirsp42.png}
    \put(17,65){\tikz\draw[red,thick] (0,0) circle (.08);}
\end{overpic}}&
\begin{overpic}[width=2.5cm]{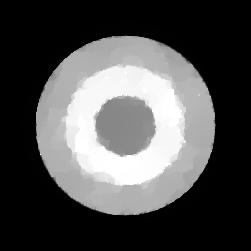}
    \put(15,64){\tikz\draw[red,thick,dashed] (0,0) rectangle (.3,.3);}
\end{overpic} &
\adjustbox{width=2.5cm,trim={.13\width} {.6\height} {0.72\width} {.25\height},clip}
{\begin{overpic}[width=2.5cm]{cirl2kl42.png}
    \put(17,65){\tikz\draw[red,thick] (0,0) circle (.08);}
\end{overpic}}  \\
& TV+SP: 21.83 & Zoomed & TV+KL+$L^2$ : 22 & Zoomed  \\[2mm]
&
\begin{overpic}[width=2.5cm]{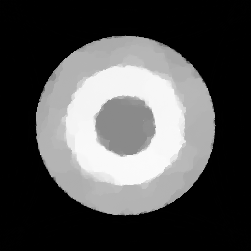}
    \put(15,64){\tikz\draw[red,thick,dashed] (0,0) rectangle (.3,.3);}
\end{overpic} &
\adjustbox{width=2.5cm,trim={.13\width} {.6\height} {0.72\width} {.25\height},clip}
{\begin{overpic}[width=2.5cm]{cirpd42.png}
    \put(17,65){\tikz\draw[red,thick] (0,0) circle (.08);}
\end{overpic}}&
\begin{overpic}[width=2.5cm]{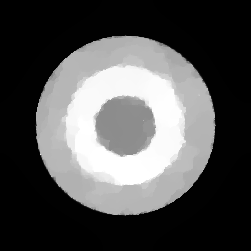}
    \put(15,64){\tikz\draw[red,thick,dashed] (0,0) rectangle (.3,.3);}
\end{overpic} &
\adjustbox{width=2.5cm,trim={.13\width} {.6\height} {0.72\width} {.25\height},clip}
{\begin{overpic}[width=2.5cm]{cirbca42.png}
    \put(17,65){\tikz\draw[red,thick] (0,0) circle (.08);}
\end{overpic}}  \\
& TV+PD: 22.16 & Zoomed & BCA : 22.18 & Zoomed  \\[2mm]
&\begin{overpic}[width=2.5cm]{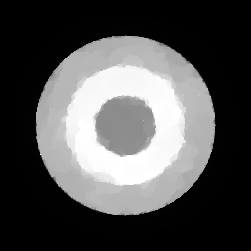}
    \put(15,64){\tikz\draw[red,thick,dashed] (0,0) rectangle (.3,.3);}
\end{overpic} &
\adjustbox{width=2.5cm,trim={.13\width} {.6\height} {0.72\width} {.25\height},clip}
{\begin{overpic}[width=2.5cm]{cirbcaf42.png}
    \put(17,65){\tikz\draw[red,thick] (0,0) circle (.08);}
\end{overpic}} 
&
\begin{overpic}[width=2.5cm]{circle4.png}
    \put(15,64){\tikz\draw[red,thick,dashed] (0,0) rectangle (.3,.3);}
\end{overpic} &
\adjustbox{width=2.5cm,trim={.13\width} {.6\height} {0.72\width} {.25\height},clip}
{\begin{overpic}[width=2.5cm]{circle4.png}
    \put(17,65){\tikz\draw[red,thick] (0,0) circle (.08);}
\end{overpic}} 
\\
& BCA$_f$ : 22.35 & Zoomed  &Truth& Zoomed
\end{tabular}
\caption{Recovery results by proposed algorithms and other compared algorithms (with SNRs(dB) below the figures)  for the image ``Circles'' in the case of MPG noises which are generated with $\eta=4, \sigma=10^{-4}$.}
\label{fig2}
\end{center}
\end{figure}

\begin{figure}[]
\begin{center}
\begin{tabular}{cclcl}
\includegraphics[width=2.5cm]{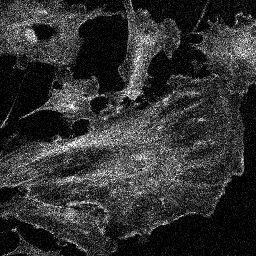} &
\begin{overpic}[width=2.5cm]{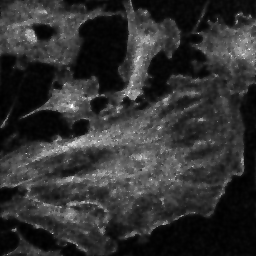}
    \put(18,48){\tikz\draw[red,thick,dashed] (0,0) rectangle (.55,.55);}
\end{overpic} &
\adjustbox{width=2.5cm,trim={.18\width} {.5\height} {.6\width} {.28\height},clip}
{\includegraphics[]{fluol162.png}}&
\begin{overpic}[width=2.5cm]{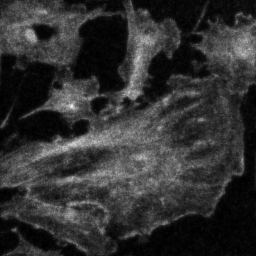}
    \put(18,48){\tikz\draw[red,thick,dashed] (0,0) rectangle (.55,.55);}
\end{overpic} &
\adjustbox{width=2.5cm,trim={.18\width} {.5\height} {.6\width} {.28\height},clip}
{\includegraphics[]{fluokl162.png}}  \\
 \makecell[tl]{\ \ \ Corruption: \\$\eta=16, \sigma=10^{-4}$ } & TV+$L^2$ : 14.17 & Zoomed & TV+KL: 14.16 & Zoomed  \\[2mm]
&
\begin{overpic}[width=2.5cm]{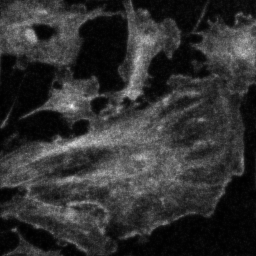}
    \put(18,48){\tikz\draw[red,thick,dashed] (0,0) rectangle (.55,.55);}
\end{overpic} &
\adjustbox{width=2.5cm,trim={.18\width} {.5\height} {.6\width} {.28\height},clip}
{\includegraphics[]{fluosp162.png}}&
\begin{overpic}[width=2.5cm]{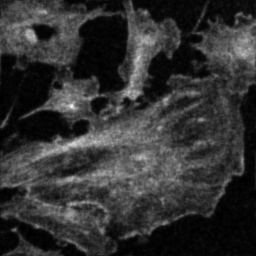}
    \put(18,48){\tikz\draw[red,thick,dashed] (0,0) rectangle (.55,.55);}
\end{overpic} &
\adjustbox{width=2.5cm,trim={.18\width} {.5\height} {.6\width} {.28\height},clip}
{\includegraphics[]{fluol2kl162.png}}  \\
& TV+SP: 14.16 & Zoomed & TV+KL+$L^2$ : 14.59 & Zoomed  \\[2mm]
&
\begin{overpic}[width=2.5cm]{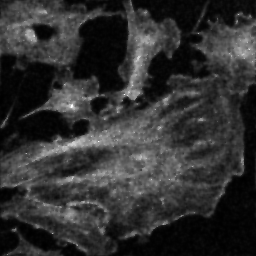}
    \put(18,48){\tikz\draw[red,thick,dashed] (0,0) rectangle (.55,.55);}
\end{overpic} &
\adjustbox{width=2.5cm,trim={.18\width} {.5\height} {.6\width} {.28\height},clip}
{\includegraphics[]{fluopd162.png}}&
\begin{overpic}[width=2.5cm]{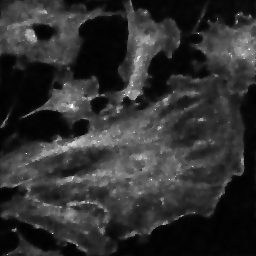}
    \put(18,48){\tikz\draw[red,thick,dashed] (0,0) rectangle (.55,.55);}
\end{overpic} &
\adjustbox{width=2.5cm,trim={.18\width} {.5\height} {.6\width} {.28\height},clip}
{\includegraphics[]{fluobca162.png}}  \\
& TV+PD: 14.47 & Zoomed & BCA : 14.42 & Zoomed  \\[2mm]
&\begin{overpic}[width=2.5cm]{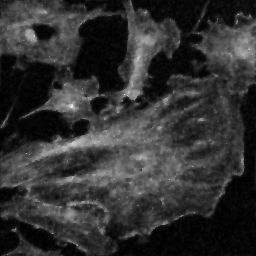}
    \put(18,48){\tikz\draw[red,thick,dashed] (0,0) rectangle (.55,.55);}
\end{overpic} &
\adjustbox{width=2.5cm,trim={.18\width} {.5\height} {.6\width} {.28\height},clip}
{\includegraphics[]{fluobcaf162.png}}
&\begin{overpic}[width=2.5cm]{fluocells1.png}
    \put(18,48){\tikz\draw[red,thick,dashed] (0,0) rectangle (.55,.55);}
\end{overpic} &
\adjustbox{width=2.5cm,trim={.18\width} {.5\height} {.6\width} {.28\height},clip}
{\includegraphics[]{fluocells1.png}}
\\
& BCA$_f$ : 14.62 & Zoomed  &Truth & Zoomed
\end{tabular}
\caption{Recovery results by proposed algorithms and other compared algorithms (with SNRs(dB) below the figures)  for the image ``Fluorescent Cells'' in the case of MPG noises which are generated with $\eta=16, \sigma=10^{-4}$.}
\label{fig21}
\end{center}
\end{figure}

\begin{figure}[]
\begin{center}
\begin{tabular}{cclcl}
\includegraphics[width=2.5cm]{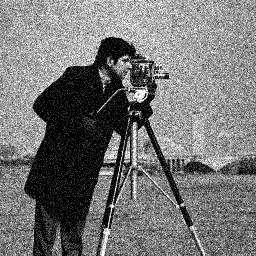} &
\begin{overpic}[width=2.5cm]{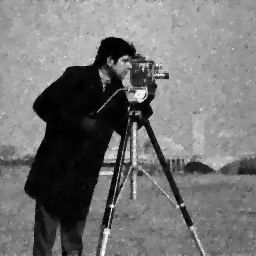}
    \put(10,9){\tikz\draw[red,thick,dashed] (0,0) rectangle (.65,.65);}
\end{overpic} &
\adjustbox{width=2.5cm,trim={.1\width} {.1\height} {.65\width} {.65\height},clip}
{\begin{overpic}[width=2.5cm]{caml641.png}
    \put(16,18){\tikz\draw[blue,thick] (0,0) circle (.1);}
\end{overpic}}&
\begin{overpic}[width=2.5cm]{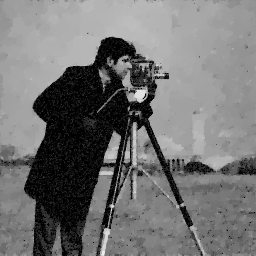}
    \put(10,9){\tikz\draw[red,thick,dashed] (0,0) rectangle (.65,.65);}
\end{overpic} &
\adjustbox{width=2.5cm,trim={.1\width} {.1\height} {.65\width} {.65\height},clip}
{\begin{overpic}[width=2.5cm]{camkl641.png}
    \put(16,18){\tikz\draw[blue,thick] (0,0) circle (.1);}
    %\put(16,18){\tikz\draw[white,thick] (0,0) circle (.15);}
\end{overpic}}  \\
 \makecell[tl]{\ \ \ Corruption: \\$\eta=64, \sigma=10^{-1}$ } & TV+$L^2$ : 20.31 & Zoomed & TV+KL: 19.78 & Zoomed  \\[2mm]
&
\begin{overpic}[width=2.5cm]{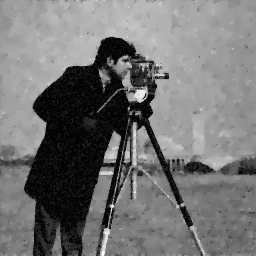}
    \put(10,9){\tikz\draw[red,thick,dashed] (0,0) rectangle (.65,.65);}
\end{overpic} &
\adjustbox{width=2.5cm,trim={.1\width} {.1\height} {.65\width} {.65\height},clip}
{\begin{overpic}[width=2.5cm]{camsp641.png}
    \put(16,18){\tikz\draw[blue,thick] (0,0) circle (.1);}
\end{overpic}}&
\begin{overpic}[width=2.5cm]{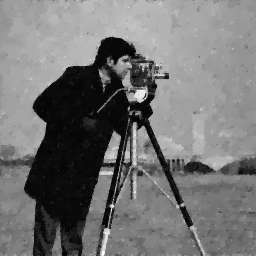}
    \put(10,9){\tikz\draw[red,thick,dashed] (0,0) rectangle (.65,.65);}
\end{overpic} &
\adjustbox{width=2.5cm,trim={.1\width} {.1\height} {.65\width} {.65\height},clip}
{\begin{overpic}[width=2.5cm]{caml2kl641.png}
    \put(16,18){\tikz\draw[blue,thick] (0,0) circle (.1);}
\end{overpic}}  \\
& TV+SP: 20.08 & Zoomed & TV+KL+$L^2$ : 20.57 & Zoomed  \\[2mm]
&
\begin{overpic}[width=2.5cm]{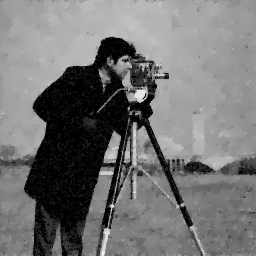}
    \put(10,9){\tikz\draw[red,thick,dashed] (0,0) rectangle (.65,.65);}
\end{overpic} &
\adjustbox{width=2.5cm,trim={.1\width} {.1\height} {.65\width} {.65\height},clip}
{\begin{overpic}[width=2.5cm]{campd641.png}
    \put(16,18){\tikz\draw[blue,thick] (0,0) circle (.1);}
\end{overpic}}&
\begin{overpic}[width=2.5cm]{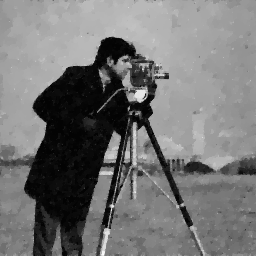}
    \put(10,9){\tikz\draw[red,thick,dashed] (0,0) rectangle (.65,.65);}
\end{overpic} &
\adjustbox{width=2.5cm,trim={.1\width} {.1\height} {.65\width} {.65\height},clip}
{\begin{overpic}[width=2.5cm]{cambca641.png}
    \put(16,18){\tikz\draw[blue,thick] (0,0) circle (.1);}
\end{overpic}}  \\
& TV+PD: 20.39 & Zoomed & BCA : 20.63 & Zoomed  \\[2mm]
&\begin{overpic}[width=2.5cm]{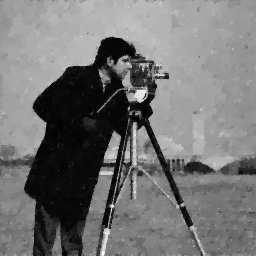}
    \put(10,9){\tikz\draw[red,thick,dashed] (0,0) rectangle (.65,.65);}
\end{overpic} &
\adjustbox{width=2.5cm,trim={.1\width} {.1\height} {.65\width} {.65\height},clip}
{\begin{overpic}[width=2.5cm]{cambcaf641.png}
    \put(16,18){\tikz\draw[blue,thick] (0,0) circle (.1);}
\end{overpic}} 
&\begin{overpic}[width=2.5cm]{cameraman.png}
    \put(10,9){\tikz\draw[red,thick,dashed] (0,0) rectangle (.65,.65);}
\end{overpic} &
\adjustbox{width=2.5cm,trim={.1\width} {.1\height} {.65\width} {.65\height},clip}
{\begin{overpic}[width=2.5cm]{cameraman.png}
    \put(16,18){\tikz\draw[blue,thick] (0,0) circle (.1);}
\end{overpic}} 
\\
& BCA$_f$ : 20.69 & Zoomed  &Truth & Zoomed
\end{tabular}
\caption{Recovery results by proposed algorithms and other compared algorithms (with SNRs(dB) below the figures)  for the image ``Cameraman'' in the case of MPG noises which are generated with $\eta=64, \sigma=10^{-1}$. Comparing with other results, there are less grey blocks in the restoration images of proposed algorithms on the blue circle, which can show the better performance of our proposed algorithms.}
\label{fig22}
\end{center}
\end{figure}

\begin{figure}[]
 \centering 
 \includegraphics[width=0.49\columnwidth]{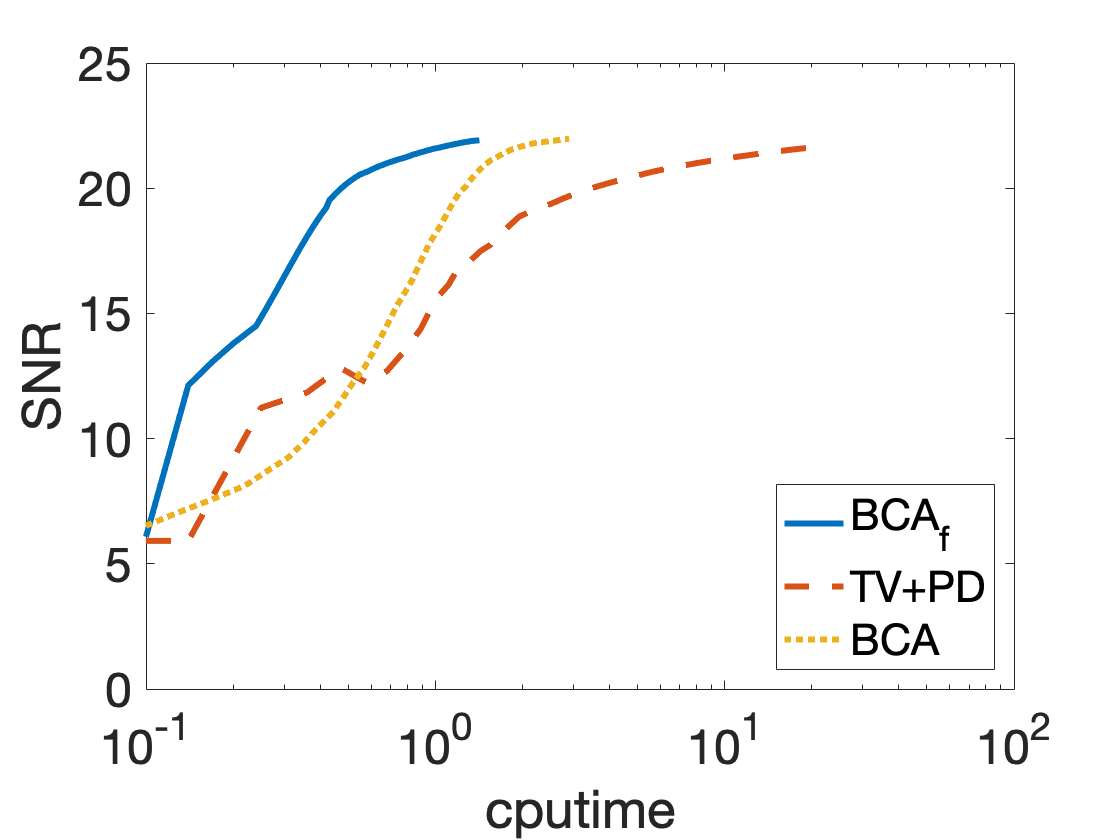}\
 \includegraphics[width=0.49\columnwidth]{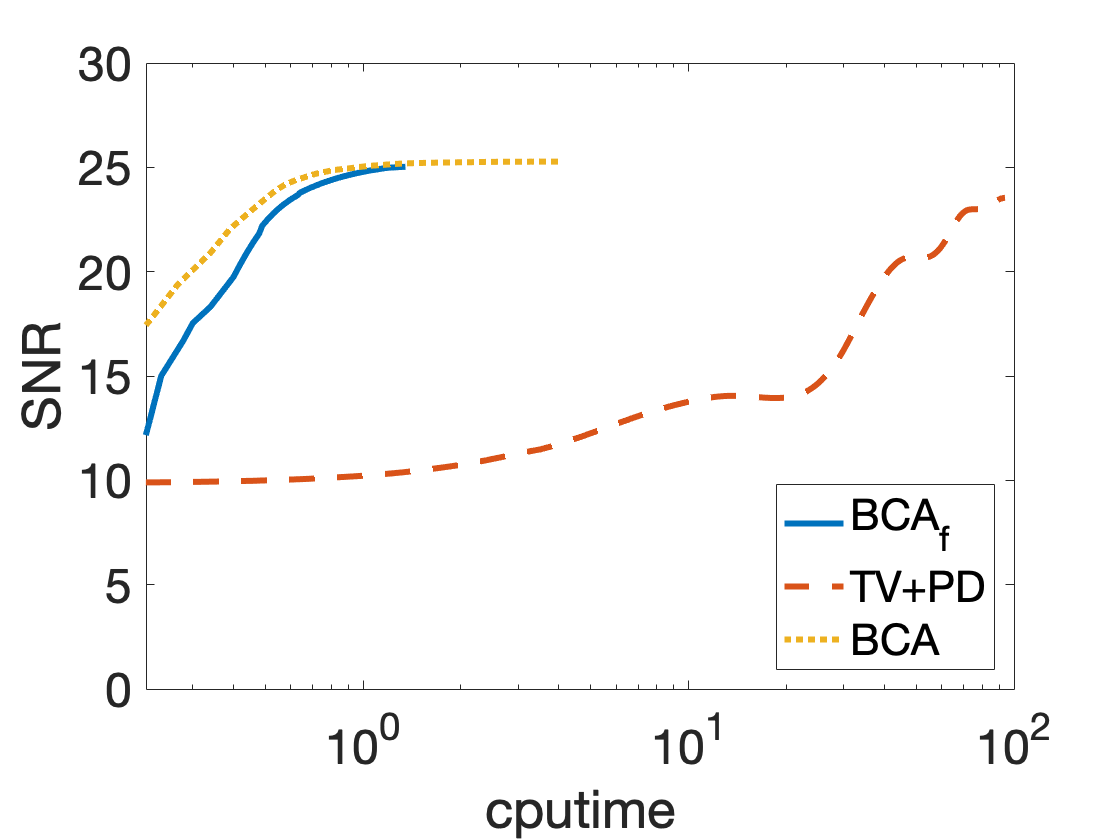}
 \caption{Histories of SNR changes for  proposed algorithms and TV+PD w.r.t. the elapsed CPU time (in log scale).  Left: $\eta=4, \sigma=10^{-1}$. Right: $\eta=16, \sigma=10^{-1}$.}
 \label{fig3}
\end{figure}

\begin{table*}[]
\twocm{\renewcommand{\arraystretch}{1.5}}
\onecm{\renewcommand{\arraystretch}{0.7}}
  \centering
  \caption{Denoising performance (First row: SNR in dB. Second row: SSIM.) with Poisson-Gaussian Noise}
  \label{tab0}
  \scalebox{.7}{\onecm{\tiny}
  \begin{tabular}{|l||c|c|c||c|c|c|c|c|c|c|}
    \hline
    Image & $\eta$ & $\sigma$ & Noisy & ~~TV+$L^2$~ & TV+KL~& TV+SP~ & TV+KL+$L^2$ & TV+PD & BCA & BCA$_f$\\\hline\hline

    \multirow{14}{*}{Circle} 
    & \multirow{2}{*}{1} & \multirow{2}{*}{$10^{-1}$} & 2.50 & 17.21 & 16.68 & 16.76 & 17.28& 17.44 & 17.84 & \bf 17.97\\\cline{4-11}
    &   &   & 0.0452 & 0.6147 & 0.4044 & 0.4059 & 0.3975 & 0.7544 & 0.7125 & \bf 0.9007\\\cline{2-11}
    & \multirow{2}{*}{1} & \multirow{2}{*}{$10^{-4}$} & 2.57 & 17.34 & 17.16 & 17.24 & 17.76 & 17.00 & 18.02 &\bf 18.10\\\cline{4-11}
    &   &   & 0.5494 & 0.6087 & 0.7997 & 0.8678 & 0.8251 & 0.8711 & 0.8913 & \bf 0.9029\\\cline{2-11}    
    & \multirow{2}{*}{4} & \multirow{2}{*}{$10^{-1}$} & 5.92 & 21.48 & 20.44 & 20.77 & 21.29 & 21.64 &\bf 22.01 & 21.78\\\cline{4-11}
    &   &   & 0.0687 & 0.7149 & 0.4763 & 0.5305 & 0.5259 & 0.9216 & 0.7153 & \bf 0.9357\\\cline{2-11}      
    & \multirow{2}{*}{4} & \multirow{2}{*}{$10^{-4}$} & 6.32 & 21.64 & 21.83 & 21.83 & 22.00 & 22.16 & 22.18 & \bf 22.35\\\cline{4-11}
    &   &   & 0.5793 & 0.7397 & 0.9385 & 0.9385 & 0.9299 & 0.9466 & \bf 0.9472 & 0.9180\\\cline{2-11}      
    & \multirow{2}{*}{16} & \multirow{2}{*}{$10^{-1}$} & 9.88 & 24.64 & 22.33 & 22.62 & 23.89 & 23.54 &\bf 25.28 & 25.04\\\cline{4-11}
    &   &   & 0.0971 & 0.8411 & 0.5088 & 0.5436 & 0.5438 & 0.8141 & \bf 0.9697 & 0.8079\\\cline{2-11}     
    & \multirow{2}{*}{16} & \multirow{2}{*}{$10^{-4}$} & 11.55 & 25.46 & 26.41 & 26.41 & 26.46 & 26.47 & 26.37 &\bf 27.12\\\cline{4-11}
    &   &   & 0.6149 & 0.8816 & 0.9500 & 0.9500 & 0.9594 & 0.9651 & \bf 0.9676 & 0.9168\\\cline{2-11}      
    & \multirow{2}{*}{Average} & \multirow{2}{*}{} & 6.46 & 21.30 & 20.81 & 20.94 & 21.45 & 21.38 & 21.95 &\bf 22.06\\\cline{4-11}
    &   &   & 0.3258 & 0.7335 & 0.6796 & 0.7061 & 0.6969 & 0.8788 & 0.8673 & \bf 0.8970\\\hline\hline

    \multirow{14}{*}{\parbox[c]{0.07\textwidth}{
        Fluorescent\\Cells}
    } 
    & \multirow{2}{*}{1} & \multirow{2}{*}{$10^{-1}$} & 1.16 & 9.88 & 9.72 & 9.72 & 9.96 & 9.48 &\bf 10.37 & 10.33\\\cline{4-11}
    &   &   & 0.0402 & 0.4861 & 0.4508 & 0.4532 & 0.4512 & 0.4572 & \bf 0.5026 & 0.4971\\\cline{2-11}     
    & \multirow{2}{*}{1} & \multirow{2}{*}{$10^{-4}$} & 1.22 & 9.97 & 9.83 & 9.83 & 9.98 & 9.79 &\bf 10.43 & 10.41\\\cline{4-11}
    &   &   & 0.0598 & 0.5058 & 0.4954 & 0.4954 & 0.5003 & 0.4471 & \bf 0.5108 & 0.5014\\\cline{2-11}     
    & \multirow{2}{*}{4} & \multirow{2}{*}{$10^{-1}$} & 3.14 & 11.10 & 11.58 & 11.54 & 11.75 & 11.62 & 11.66 &\bf 12.06\\\cline{4-11}
    &   &   & 0.1181 & 0.5554 & 0.5369 & 0.5239 & 0.5588 & 0.5674 & 0.5753 & \bf 0.5801\\\cline{2-11}      
    & \multirow{2}{*}{4} & \multirow{2}{*}{$10^{-4}$} & 3.59 & 11.25 & 11.88 & 11.88 & 12.17 & 11.88 & 11.96 &\bf 12.38\\\cline{4-11}
    &   &   & 0.1680 & 0.5765 & 0.6078 & 0.6078 & 0.6133 & 0.5531 & \bf 0.6160 & 0.6139\\\cline{2-11}       
    & \multirow{2}{*}{16} & \multirow{2}{*}{$10^{-1}$} & 5.77 & 13.43 & 12.66 & 12.64 & 13.27 & 13.45 & 13.37 &\bf 13.50\\\cline{4-11}
    &   &   & 0.2282 & 0.6424 & 0.5998 & 0.5989 & 0.6383 & 0.6669 & 0.6557 & \bf 0.6685\\\cline{2-11}     
    & \multirow{2}{*}{16} & \multirow{2}{*}{$10^{-4}$} & 7.87 & 14.17 & 14.16 & 14.16 & 14.59 & 14.47 & 14.42 &\bf 14.62 \\\cline{4-11}
    &   &   & 0.4003 & 0.7360 & 0.7228 & 0.7228 & 0.7388 & 0.7309 & \bf 0.7379 & 0.7368\\\cline{2-11}      
    & \multirow{2}{*}{Average} & \multirow{2}{*}{} & 3.79 & 11.63 & 11.64 & 11.63 & 11.95 & 11.78 & 12.04 &\bf 12.22\\\cline{4-11}
    &   &   & 0.1691 & 0.5837 & 0.5689 & 0.5670 & 0.5835 & 0.5704 & \bf 0.5997 & 0.5996\\\hline\hline    

    \multirow{14}{*}{Cameraman} 
    & \multirow{2}{*}{1} & \multirow{2}{*}{$10^{-1}$} & 1.97 & 13.06 & 14.25 & 14.19 & 14.30 & 14.30 &\bf 14.59 & 14.56\\\cline{4-11}
    &   &   & 0.0496 & 0.4167 & 0.5498 & 0.5322 & 0.5432 & 0.5524 & 0.5633 & \bf 0.5644\\\cline{2-11}      
    & \multirow{2}{*}{1} & \multirow{2}{*}{$10^{-4}$} & 2.00 & 13.09 & 14.36 & 14.36 & 14.40 & 14.33 &\bf 14.57 & 14.52\\\cline{4-11}
    &   &   & 0.0628 & 0.4238 & 0.4602 & 0.4602 & 0.4760 & 0.4362 & \bf 0.5854 & 0.5659\\\cline{2-11}       
    & \multirow{2}{*}{4} & \multirow{2}{*}{$10^{-1}$} & 5.02 & 15.66 & 16.46 & 16.43 & 16.56 & 16.17 & 16.79 &\bf 16.83\\\cline{4-11}
    &   &   & 0.1178 & 0.5729 & 0.6552 & 0.6540 & 0.6720 & 0.6422 & 0.6417 & \bf 0.6655\\\cline{2-11}  
    & \multirow{2}{*}{4} & \multirow{2}{*}{$10^{-4}$} & 5.28 & 15.81 & 16.27 & 16.27 & 16.47 & 16.14 & 16.99 &\bf 17.00\\\cline{4-11}
    &   &   & 0.1514 & 0.5879 & 0.6301 & 0.6301 & 0.6160 & 0.6047 & 0.6697 & \bf 0.6770\\\cline{2-11}  
    & \multirow{2}{*}{16} & \multirow{2}{*}{$10^{-1}$} & 9.06 & 18.25 & 18.60 & 18.60 & 18.81 & 18.18 & 18.99 &\bf 19.02\\\cline{4-11}
    &   &   & 0.1992 & 0.6393 & 0.7126 & 0.7181 & 0.7163 & 0.6527 & 0.7112 & \bf 0.7214\\\cline{2-11}  
    & \multirow{2}{*}{16} & \multirow{2}{*}{$10^{-4}$} & 10.24 & 18.88 & 19.44 & 19.44 & 19.64 & 19.59 &\bf 19.81 & 19.53\\\cline{4-11}
    &   &   & 0.2920 & 0.6980 & 0.7321 & 0.7321 & 0.7503 & 0.7317 & 0.7614 & \bf 0.7764\\\cline{2-11}  
    & \multirow{2}{*}{Average} & \multirow{2}{*}{} & 5.60 & 15.79 & 16.56 & 16.55 & 16.70 & 16.45 &\bf 16.96 & 16.91\\\cline{4-11}
    &   &   & 0.1455 & 0.5564 & 0.6233 & 0.6211 & 0.6290 & 0.6033 & 0.6555 & \bf 0.6618\\\hline

  \end{tabular}
}
\end{table*}

\begin{table*}[]
\twocm{\renewcommand{\arraystretch}{1.5}}
\onecm{\renewcommand{\arraystretch}{0.7}}
  \centering
  \caption{Computational time of the proposed algorithms and TV+PD (in seconds)}
  \label{tab1}
  \scalebox{.75}{\onecm{\tiny}
  \begin{tabular}{|l||c|c||c|c|c|}
    \hline
    Image & $\eta$ & $\sigma$ & ~~TV+PD & BCA & BCA$_f$\\\hline\hline

    \multirow{7}{*}{Circle} 
    & 1 & $10^{-1}$ & 40.4971 & 4.6052 &\bf 1.0314\\\cline{2-6}
    & 1 & $10^{-4}$ & 18.5340 & 4.6922 &\bf 1.2650\\\cline{2-6}
    & 4 & $10^{-1}$ & 7.3436 & 0.7641 &\bf 0.7001\\\cline{2-6}
    & 4 & $10^{-4}$ & 39.6196 & 1.1891 &\bf 0.6030\\\cline{2-6}
    & 16 & $10^{-1}$ & 36.3221 & 1.2652 &\bf 0.4643\\\cline{2-6}
    & 16 & $10^{-4}$ & 39.5773 & 0.6123 &\bf 0.3070\\\cline{2-6}
    & Average & & 30.3156 & 2.1880 &\bf 0.7285\\\hline\hline

    \multirow{7}{*}{\parbox[c]{0.07\textwidth}{
        Fluorescent\\Cells}
    } 
    & 1 & $10^{-1}$ & 21.7999 & 8.3016 &\bf 2.1735\\\cline{2-6}
    & 1 & $10^{-4}$ & 41.6892 & 8.9243 &\bf 2.5687\\\cline{2-6}
    & 4 & $10^{-1}$ & 40.1406 & 4.1263 &\bf 2.4025\\\cline{2-6}
    & 4 & $10^{-4}$ & 38.6224 & 3.8464 &\bf 2.0889\\\cline{2-6}
    & 16 & $10^{-1}$ & 37.7174 & 7.3265 &\bf 6.1742\\\cline{2-6}
    & 16 & $10^{-4}$ & 7.7788 & 7.6185 &\bf 1.1258 \\\cline{2-6}
    & Average & & 31.2914 & 6.6906 &\bf 2.7556\\\hline\hline

    \multirow{7}{*}{Cameraman} 
    & 1 & $10^{-1}$ & 9.3735 & 1.3062 &\bf 1.2272\\\cline{2-6}
    & 1 & $10^{-4}$ & 40.9848 & 3.2183 &\bf 0.9053\\\cline{2-6}
    & 4 & $10^{-1}$ & 36.8623 &\bf 0.6872 & 1.5084\\\cline{2-6}
    & 4 & $10^{-4}$ & 36.4358 &\bf 0.6179 & 1.4209\\\cline{2-6}
    & 16 & $10^{-1}$ & 3.3691 &\bf 0.5151 & 0.9925\\\cline{2-6}
    & 16 & $10^{-4}$ & 11.0407 &\bf 0.3607 & 0.5260\\\cline{2-6}
    & Average & & 23.0140 & 1.1176 &\bf 1.0967\\\hline

  \end{tabular}
}
\end{table*}

To show the convergence of BCA and BCA$_f$, we plot the convergence curves in Fig. \ref{fig4} ($\eta=1$ and $\sigma=10^{-1},10^{-4}$). One can readily see that the errors (SE) are quite steadily decreasing, demonstrating the convergence of the proposed algorithms.

\begin{figure}[]
 \centering 
 \includegraphics[width=0.49\columnwidth]{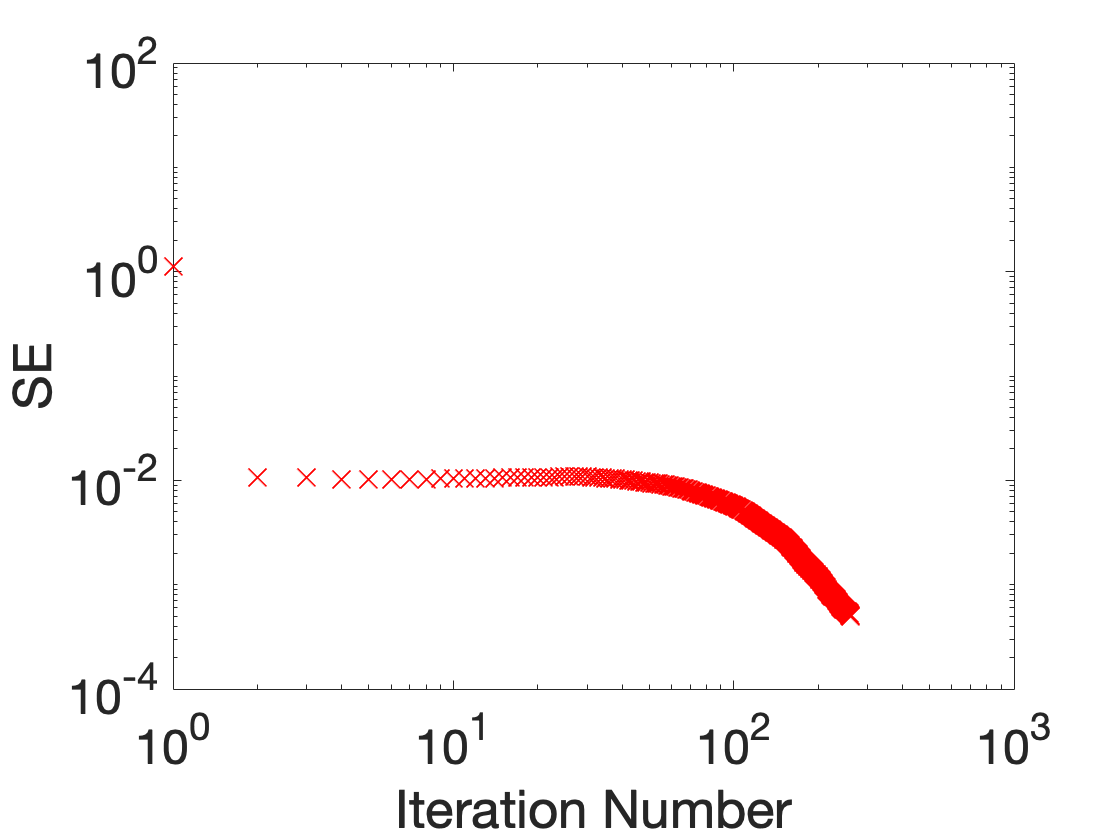}\
 \includegraphics[width=0.49\columnwidth]{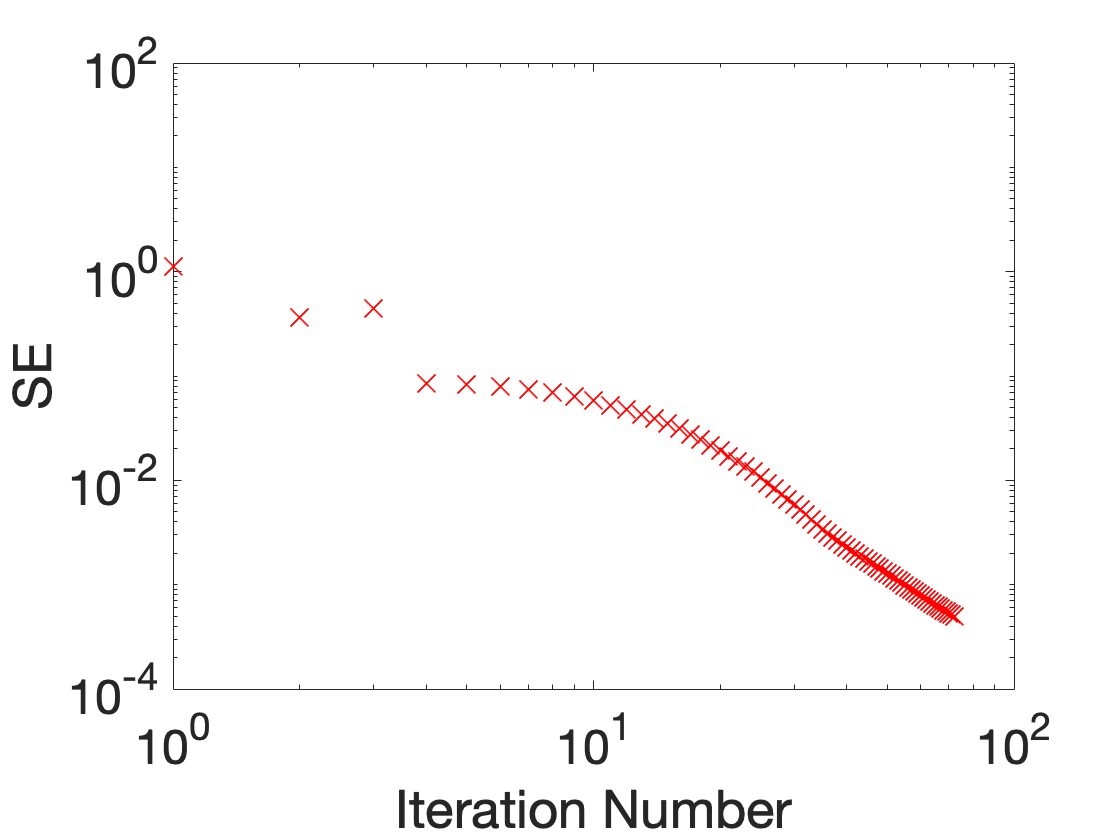}\\
 \includegraphics[width=0.49\columnwidth]{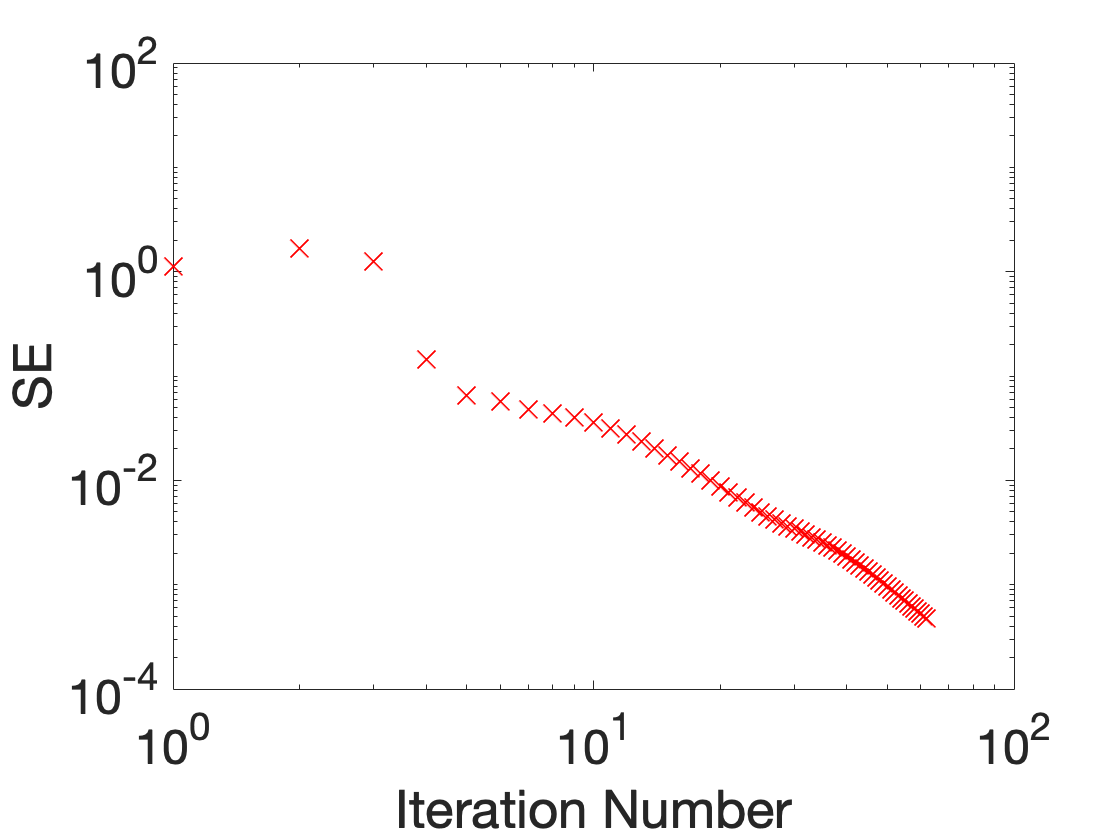}\
 \includegraphics[width=0.49\columnwidth]{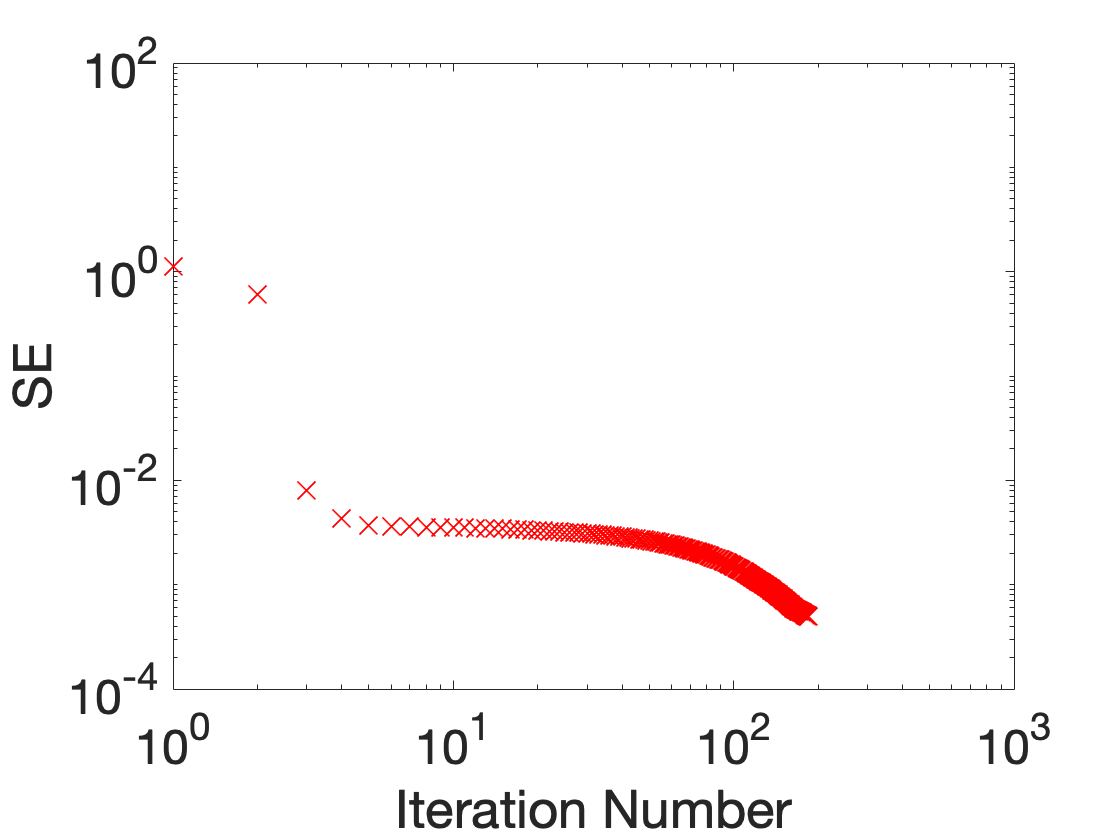}
 \caption{SE ($\tfrac{\Vert u_{k+1}-u_{k}\Vert}{\Vert u_{k}\Vert}$) changes  v.s. iteration number (both in log scale): Top: BCA; Bottom: BCA$_f$; Left: $\eta=1, \sigma=10^{-1}$; Right: $\eta=1, \sigma=10^{-4}$. The test image is  Circles(Fig\ref{fig1}(a)).}
 \label{fig4}
\end{figure}

\subsection{Robustness w.r.t.  the parameters}
One readily knows that the model parameters $\lambda_1$ and $\lambda_2$ are critical to recovery results, which essentially balance the data fitting terms and the regularization terms. Here we only study how the algorithm parameters $\alpha$, $\alpha_w$ and $\alpha_p$ affect  the performance of proposed BCA and BCA$_f$ respectively.
In Fig. \ref{fig5}, the SNR changes w.r.t. $\alpha$ for BCA algorithm is illustrated, in which $\lambda_1$ and $\lambda_2$ are fixed, and $\alpha$ varies from 20 to 2000. It is obviously that BCA is quite robust to the parameter $\alpha$.
As for BCA$_f$, the SNR changes w.r.t. different parameters are illustrated in Fig. \ref{fig6}(a), in which $\lambda_1, \lambda_2$ and $\alpha_p$ are fixed, and $\alpha_w$ varies from 10 to 4000. Similarly, the SNR changes are put in Fig. \ref{fig6}(b), where $\alpha_p$ varies from 10 to 5000. One can easily see that BCA$_f$ is quite robust to the parameter $\alpha_p$, while more sensitive to parameter $\alpha_w$.

\begin{figure}[]
 \centering 
 \includegraphics[width=0.49\columnwidth]{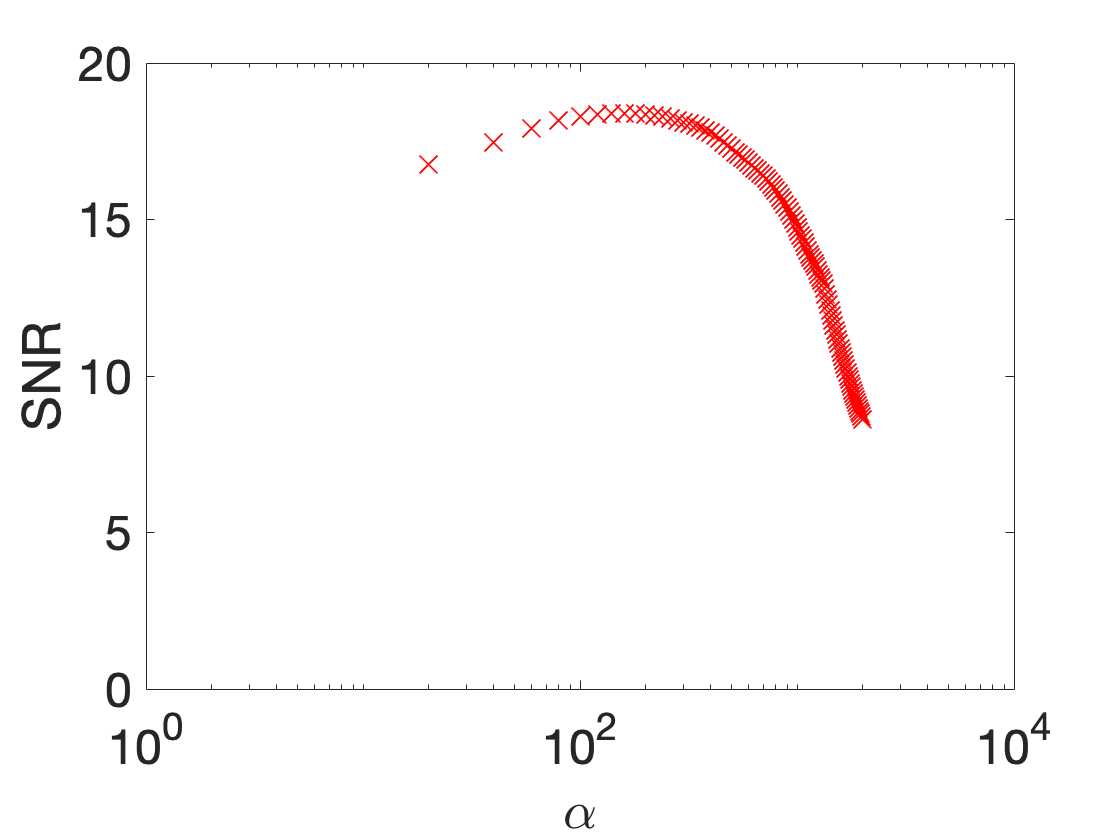}
 \caption{The performance of BCA w.r.t. $\alpha$ (in log scale), where $\eta=1, \sigma=10^{-4}$, using test image  Circles(Fig\ref{fig1}(a)). }
 \label{fig5}
\end{figure}

\begin{figure}[]
 \centering 
 \subfloat[]{\includegraphics[width=0.49\columnwidth]{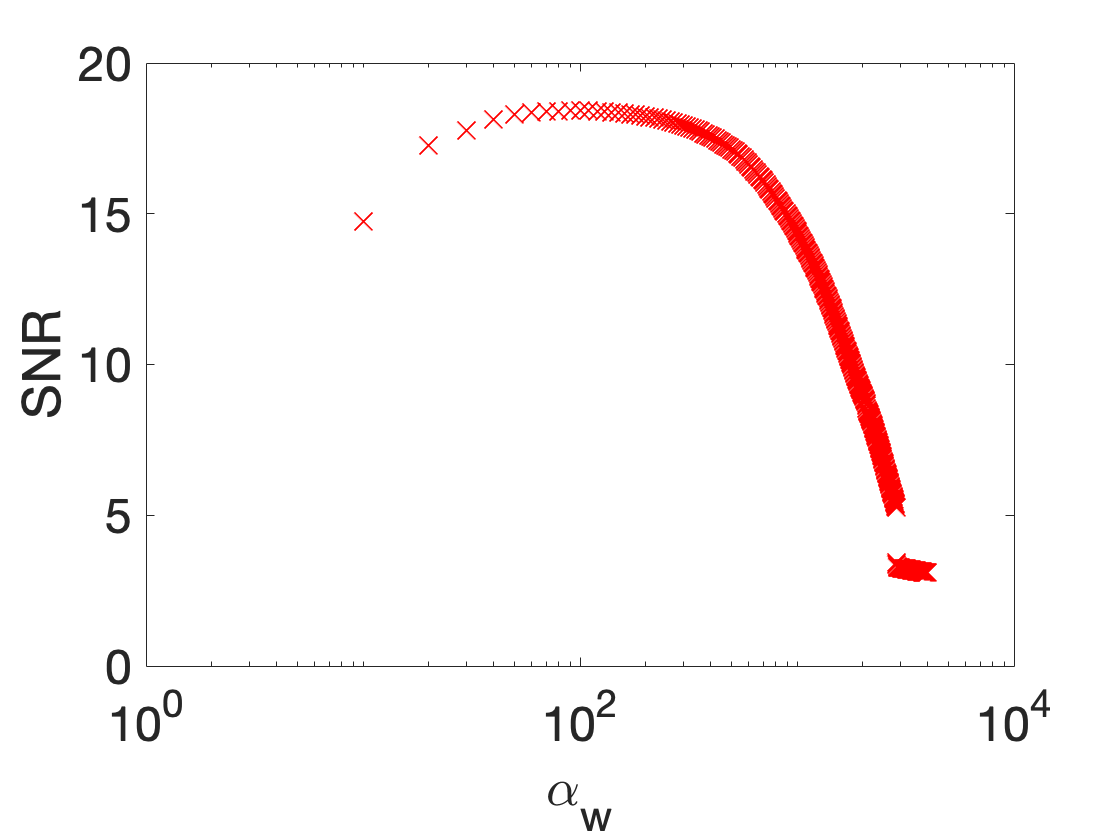}}\
 \subfloat[]{\includegraphics[width=0.49\columnwidth]{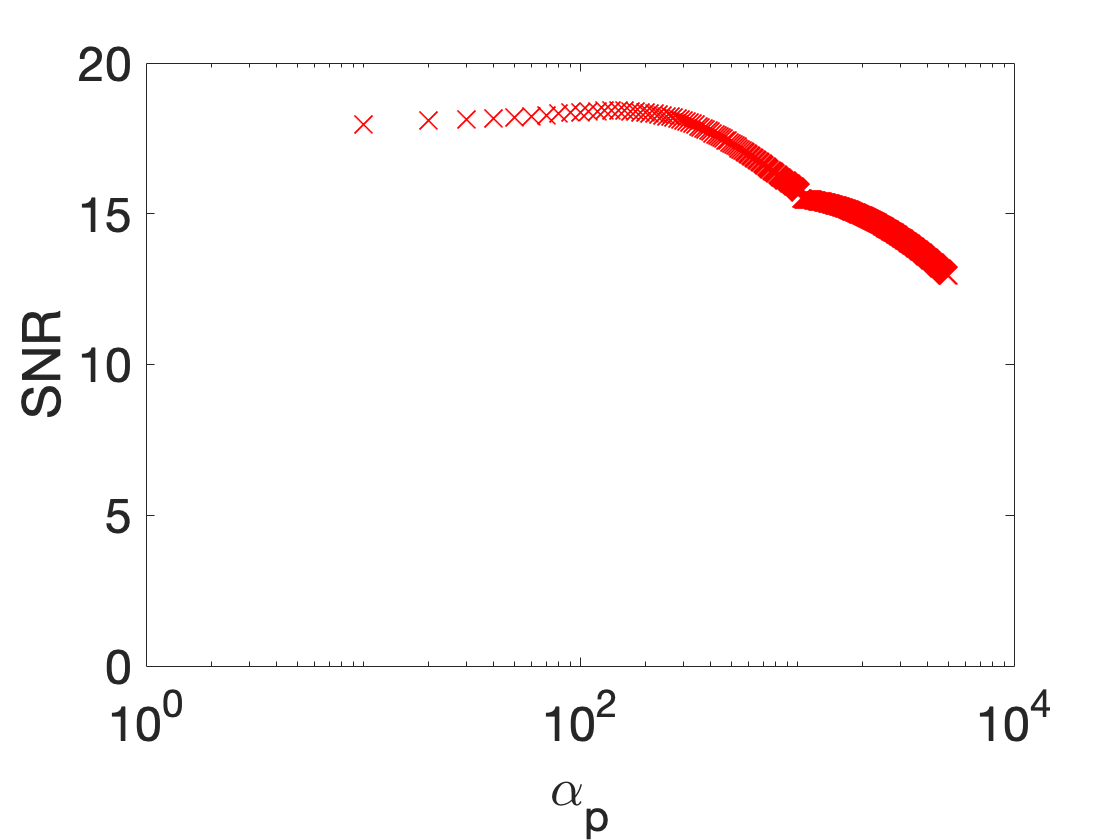}}
 \caption{The performance of BCA$_f$ w.r.t. $\alpha_w$ (in log scale), $\alpha_p$ (in log scale), where $\eta=1, \sigma=10^{-4}$, using test image  Circles(Fig\ref{fig1}(a))}
 \label{fig6}
\end{figure}

We introduce a positive scalar $\epsilon$  for convergence guarantee of proposed BCA algorithm. To demonstrate its reasonableness, we show the performances changes w.r.t. this parameter, and put the SNRs changes of recovery results  by BCA algorithm in Fig. \ref{fig9}. One can readily observe that when $\epsilon$ lies in the range from $10^{-10}$ to $10^{-1}$, the SNR value is quite stable. 

\begin{figure}[]
 \centering 
 \subfloat[]{\includegraphics[width=0.3\columnwidth]{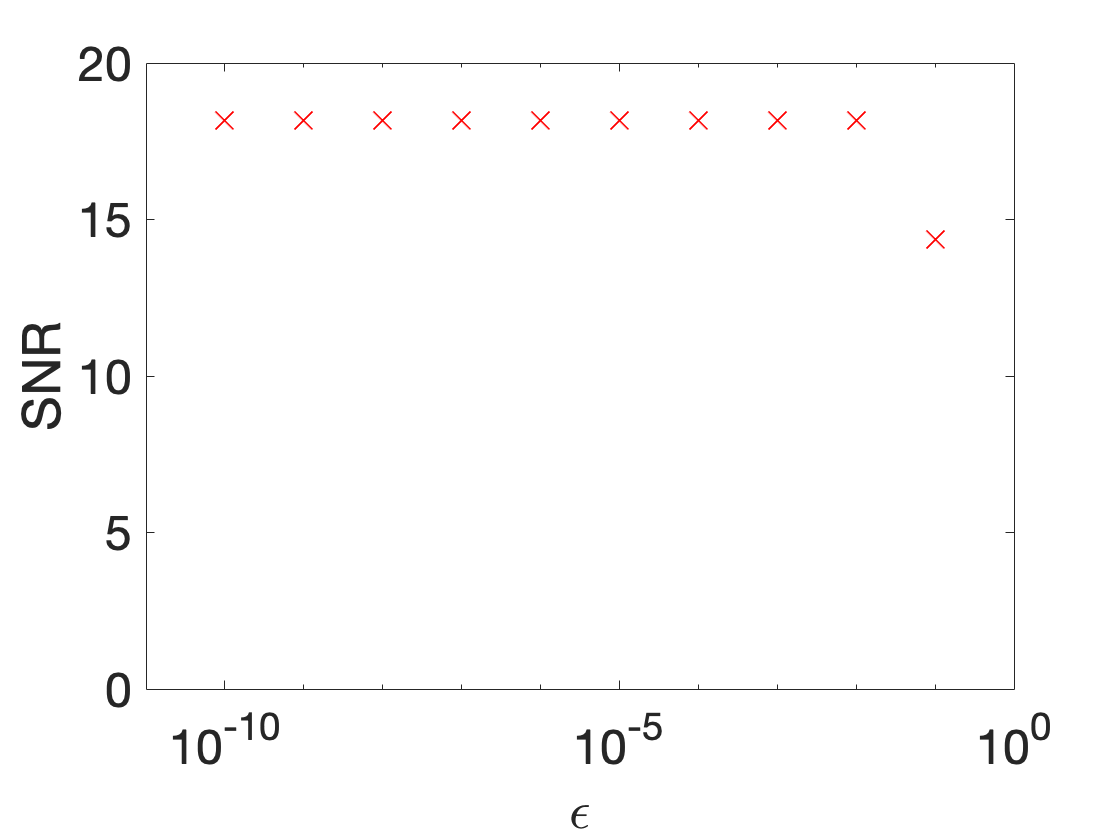}}\
 \subfloat[]{\includegraphics[width=0.3\columnwidth]{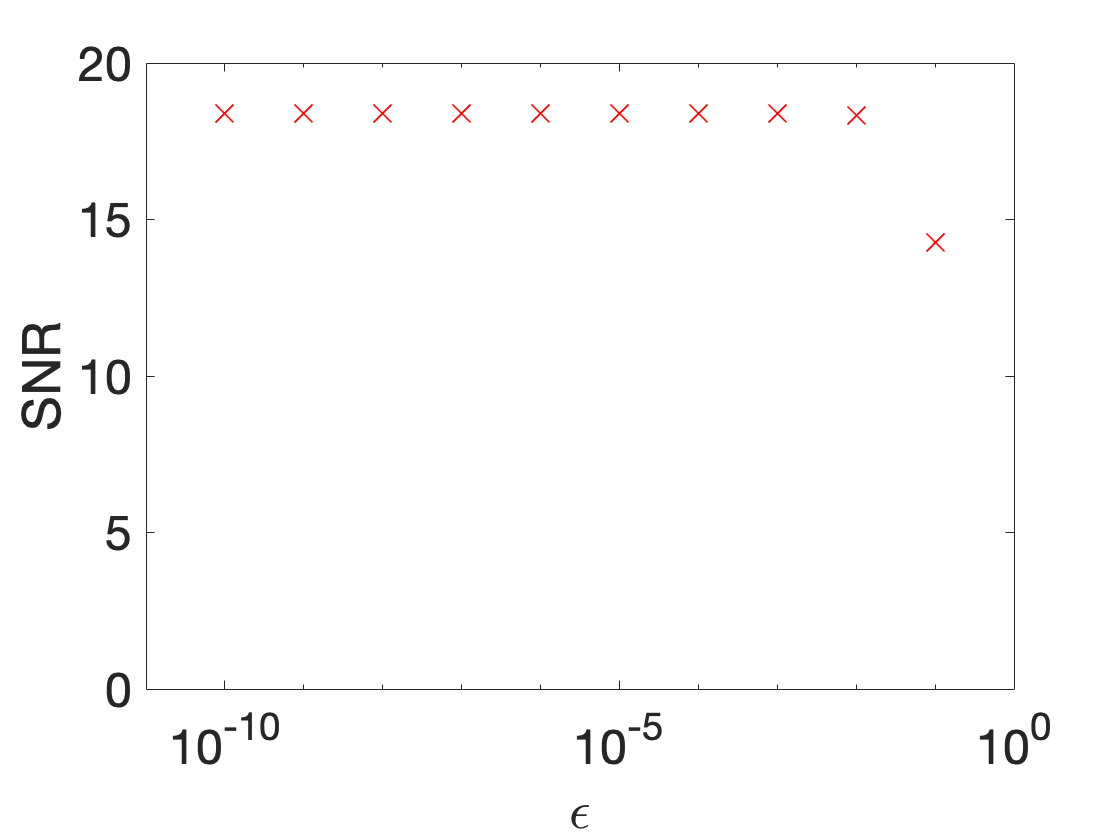}}\
 \subfloat[]{\includegraphics[width=0.3\columnwidth]{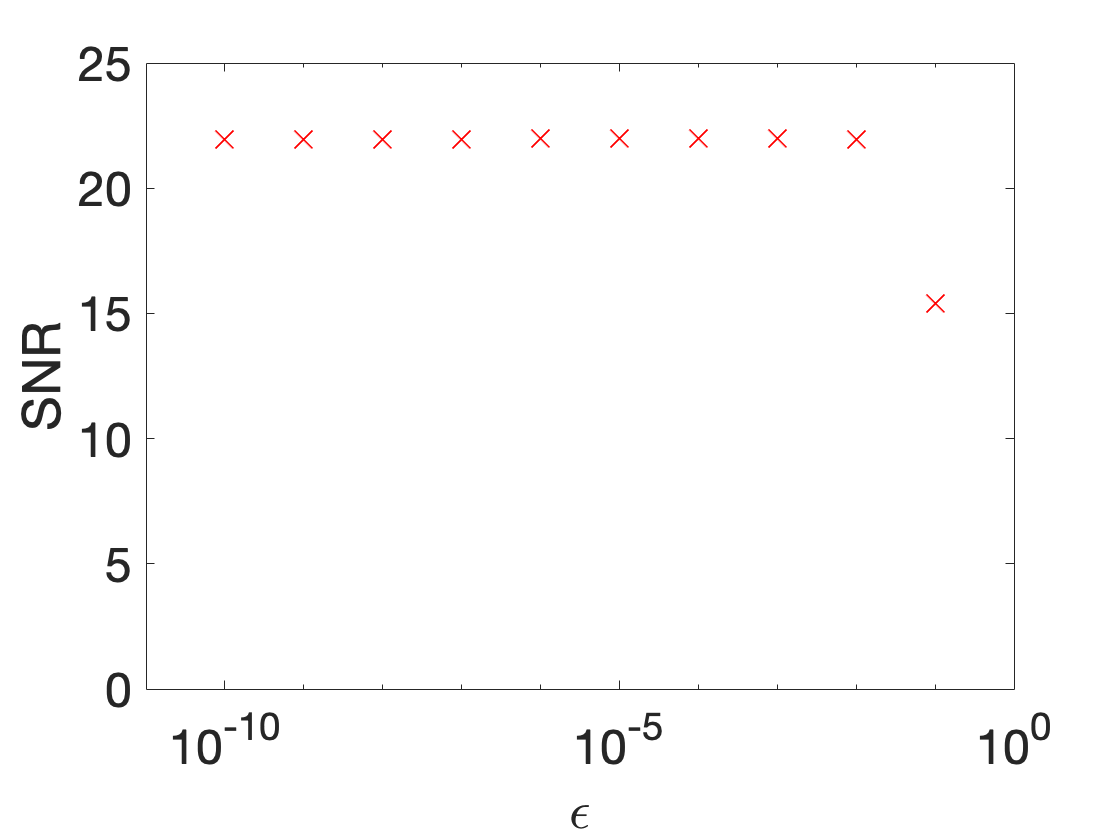}}\\
 \subfloat[]{\includegraphics[width=0.3\columnwidth]{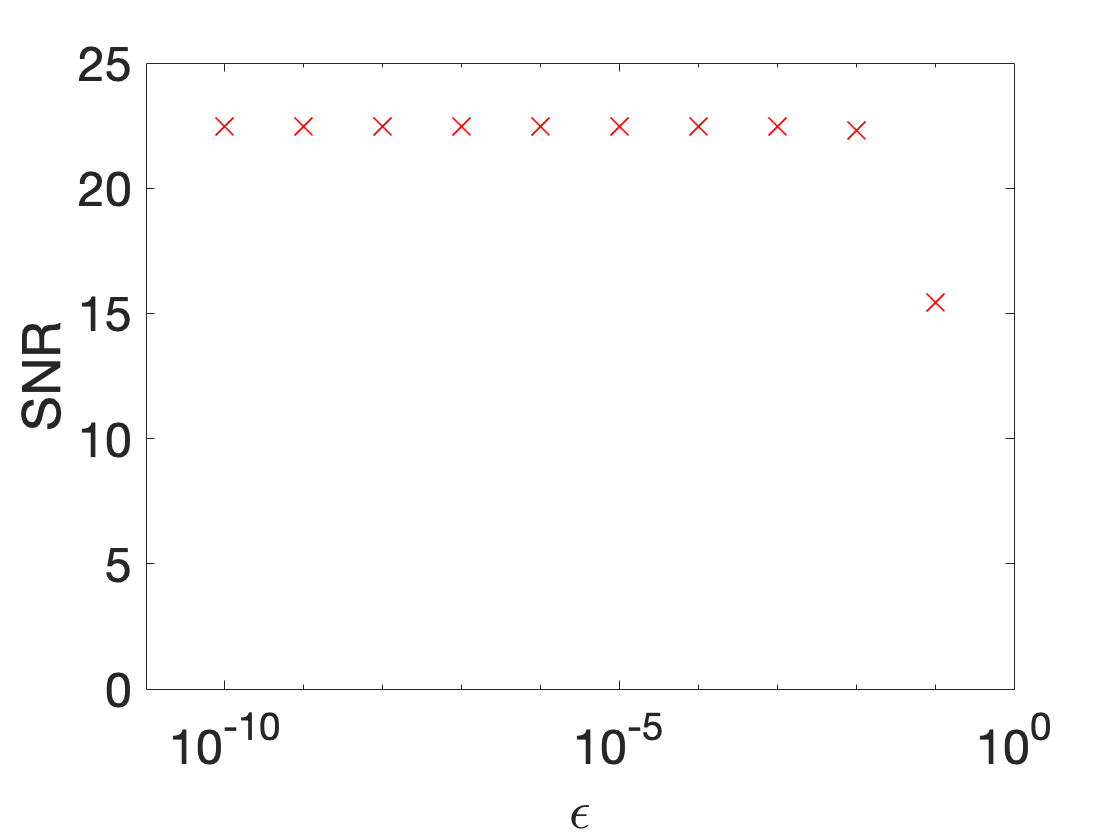}}\
 \subfloat[]{\includegraphics[width=0.3\columnwidth]{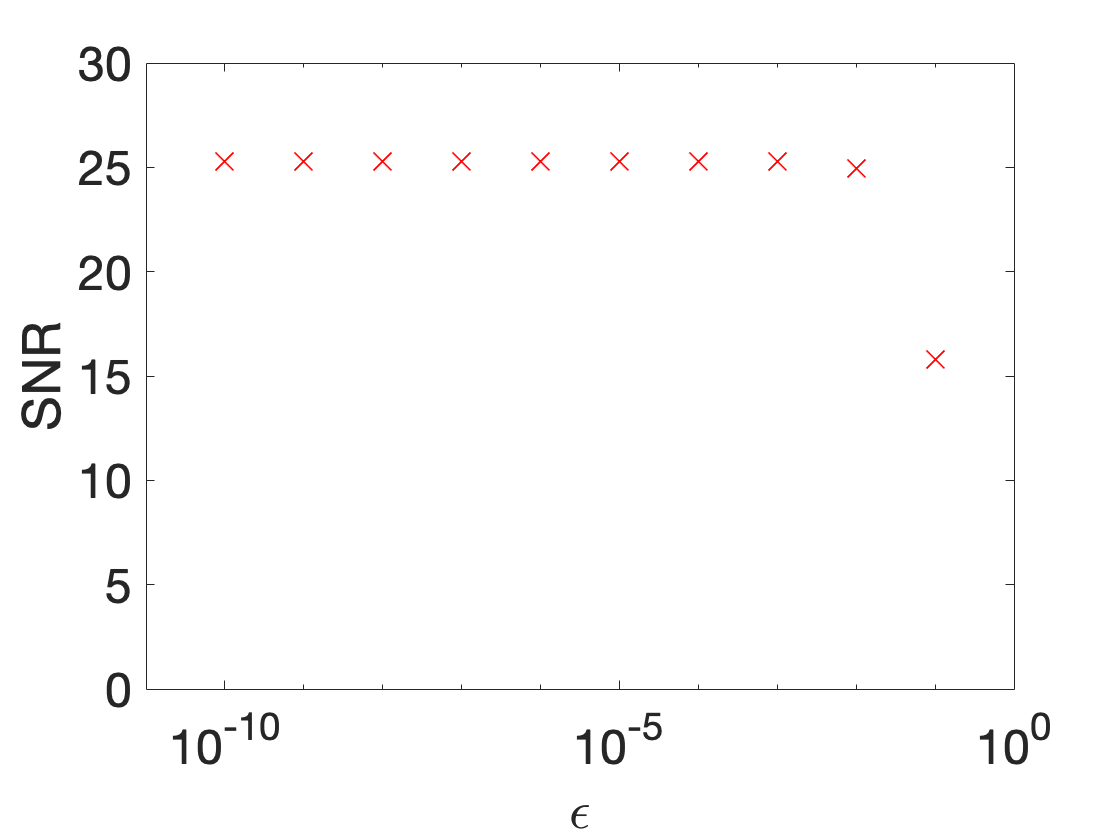}}\
 \subfloat[]{\includegraphics[width=0.3\columnwidth]{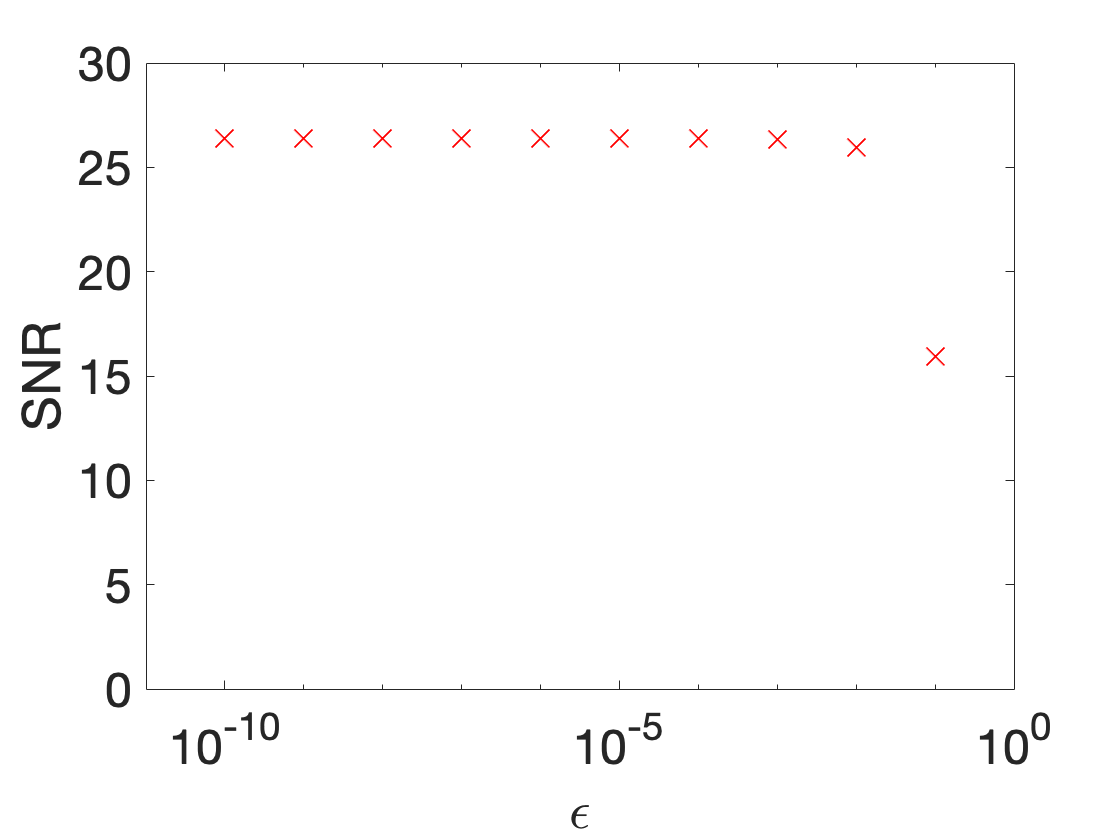}}
 \caption{SNR changes w.r.t.  the parameter $\epsilon$ for BCA Algorithm, using test image Circles(Fig\ref{fig1}(a)).}
 \label{fig9}
\end{figure}

\subsection{Numerical validations}
To validate the Assumption \ref{asum}  numerically, which is used for the convergence analysis of the proposed BCA, we plot the minimum value curves of the iterative solutions  $\{w^k\}$ (see Fig. \ref{fig7}), which clearly show that the minimum value of $w$  during iteration are almost bigger than 0.2, which show the reasonableness of this assumption.

\begin{figure}[]
 \centering 
 \subfloat[]{\includegraphics[width=0.3\columnwidth]{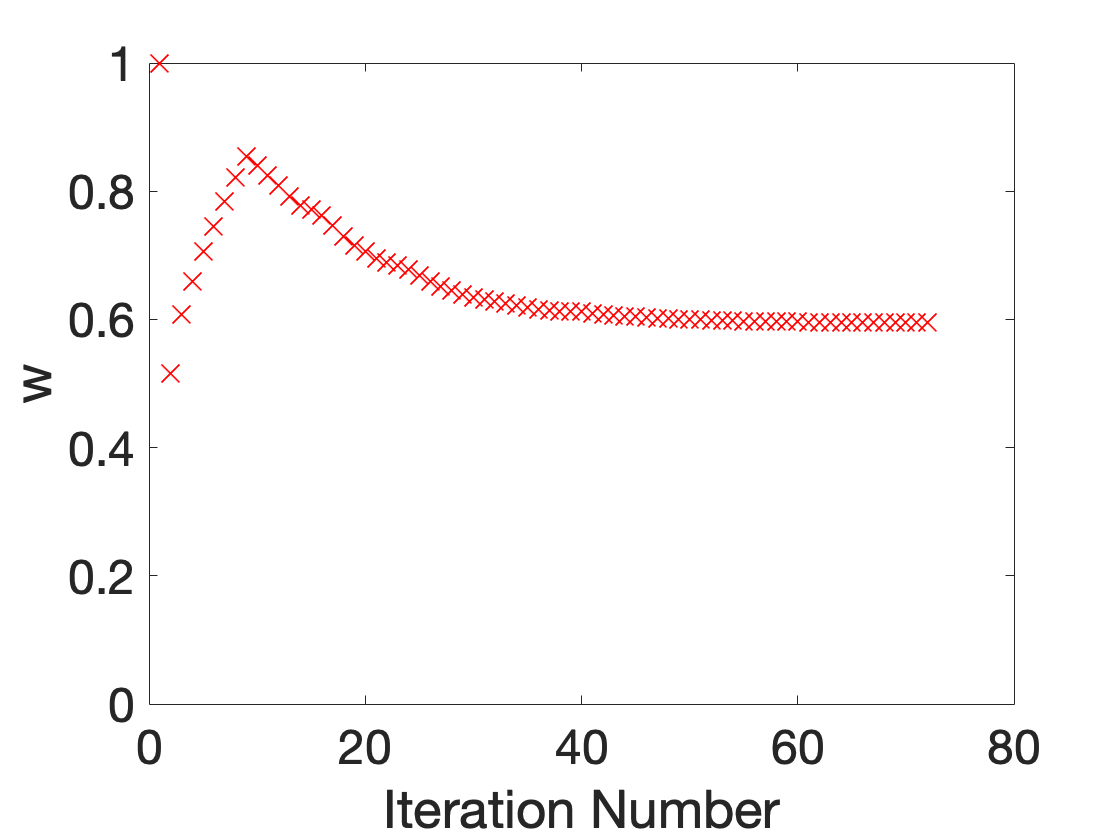}}\
 \subfloat[]{\includegraphics[width=0.3\columnwidth]{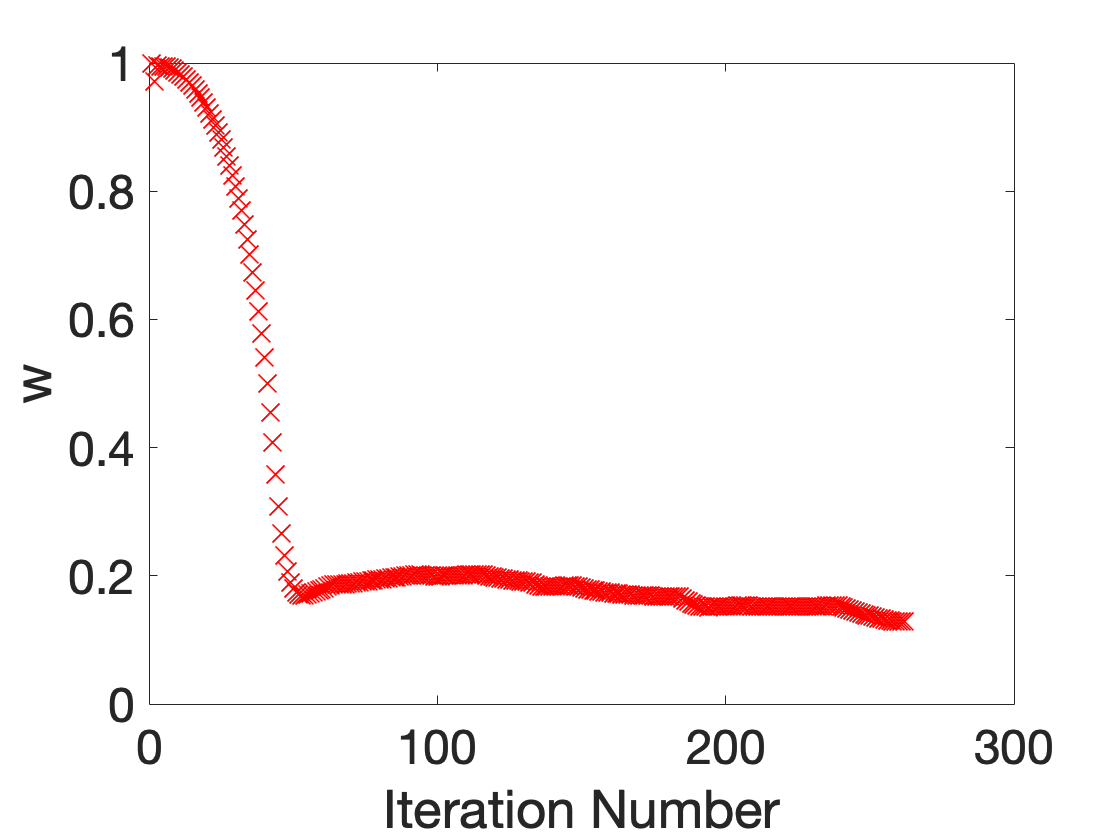}}\
 \subfloat[]{\includegraphics[width=0.3\columnwidth]{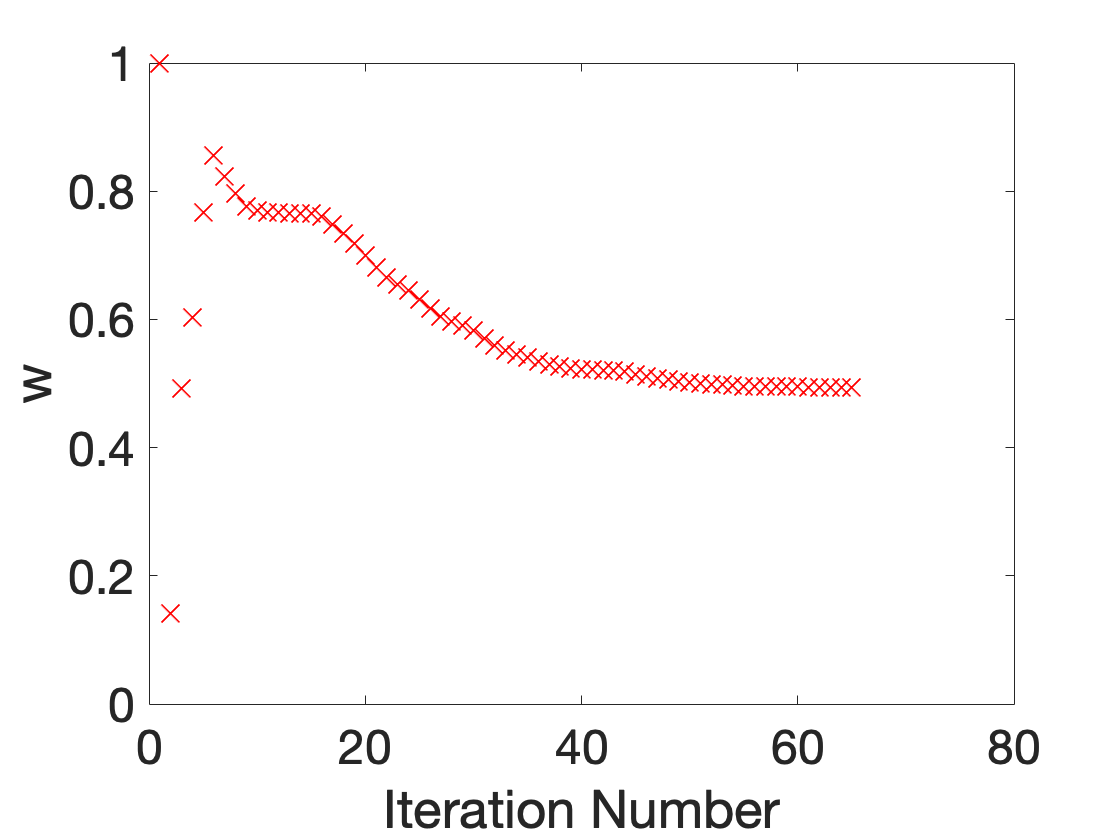}}\\
 \subfloat[]{\includegraphics[width=0.3\columnwidth]{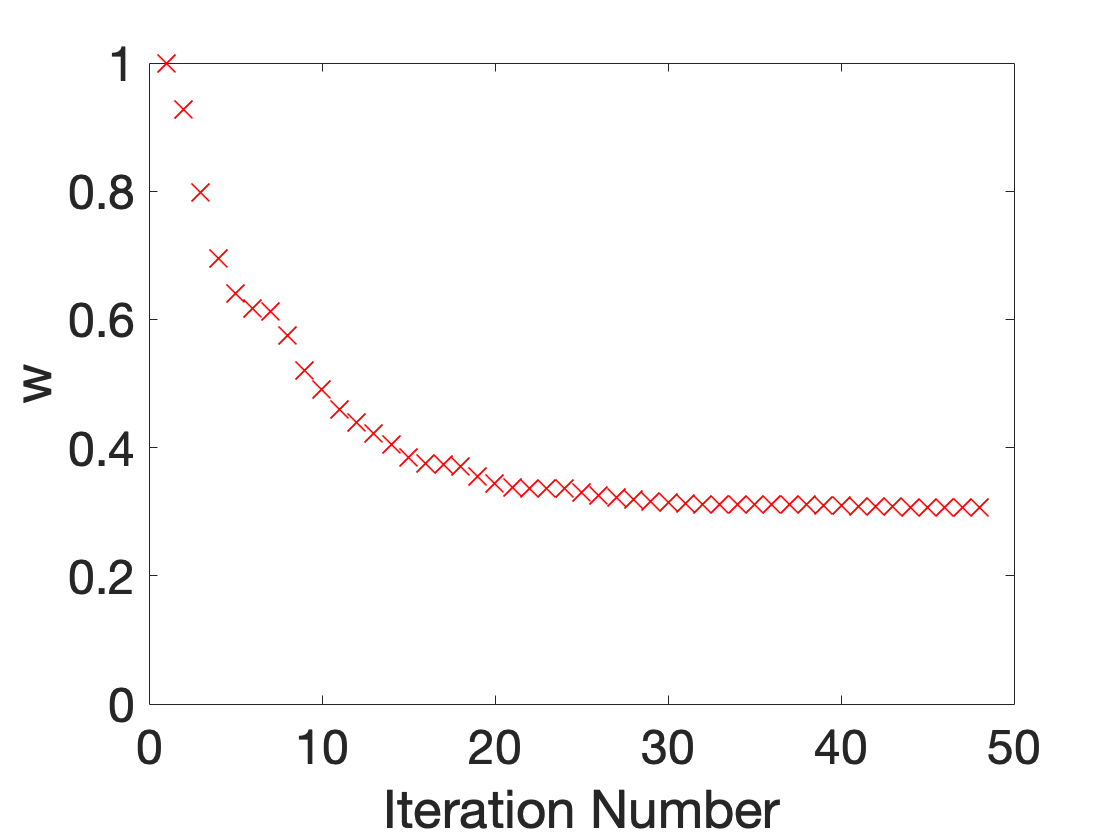}}\
 \subfloat[]{\includegraphics[width=0.3\columnwidth]{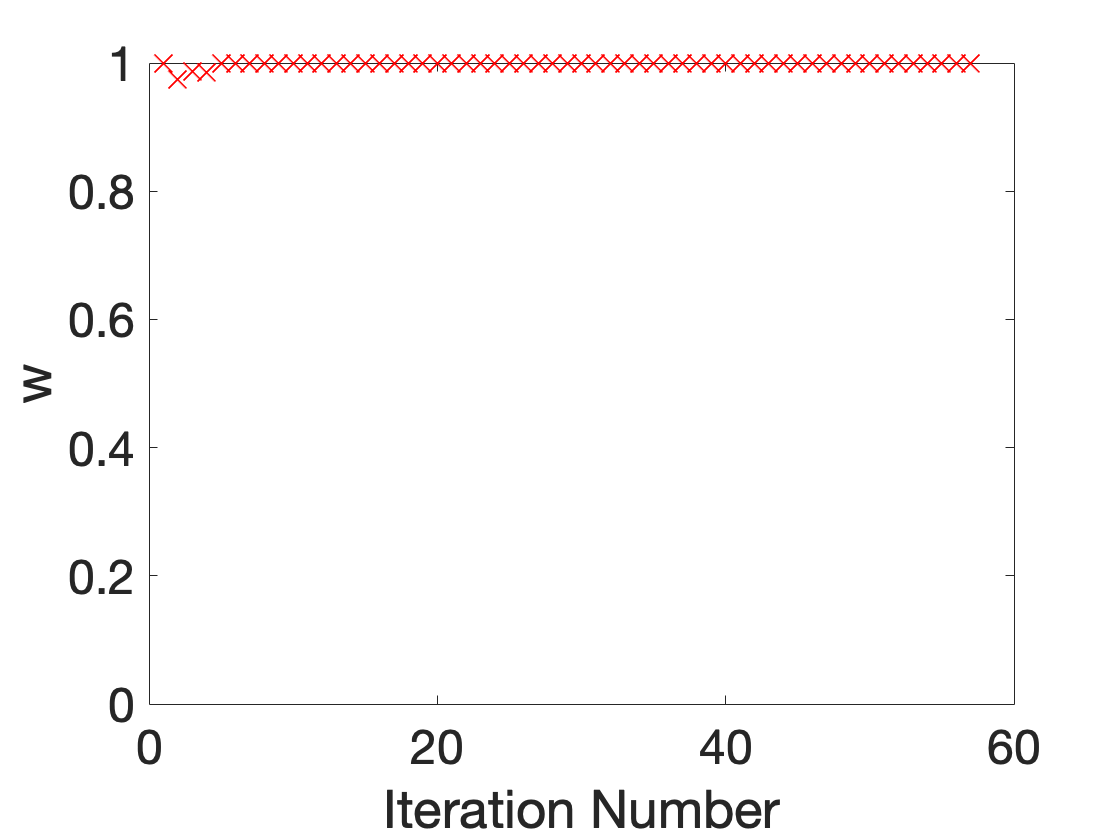}}\
 \subfloat[]{\includegraphics[width=0.3\columnwidth]{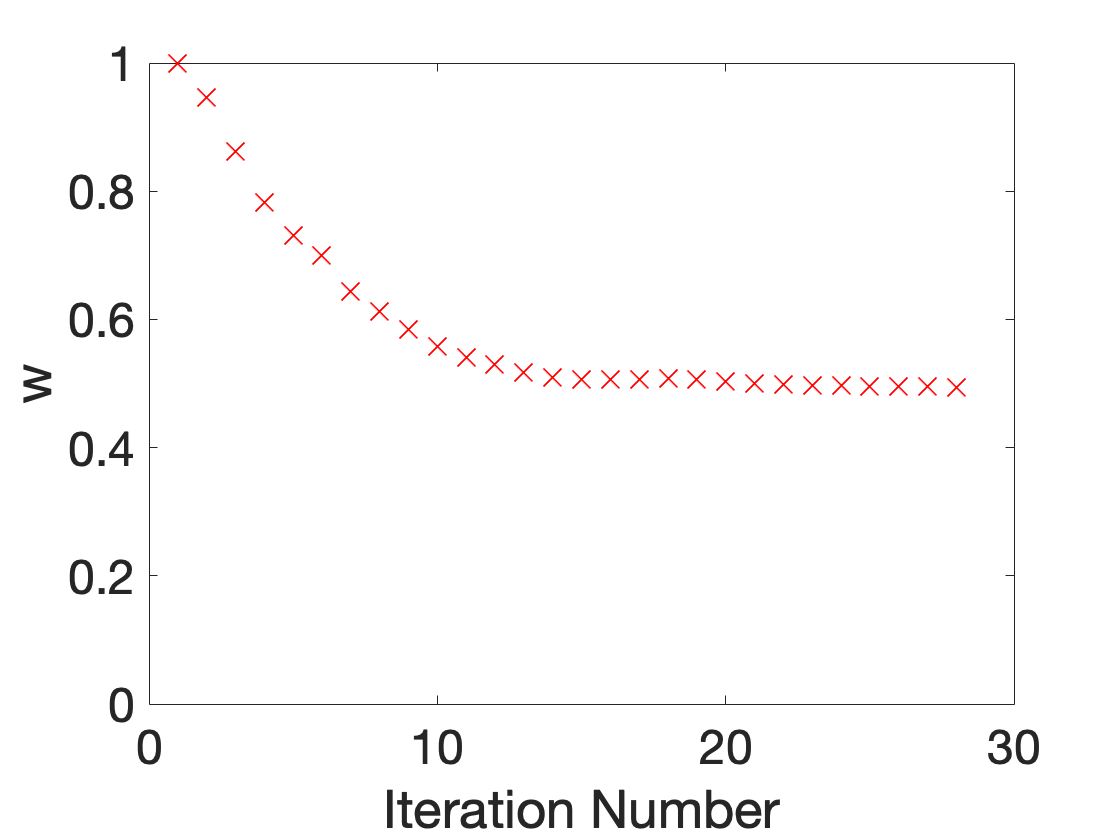}}
 \caption{Minimum value curves of $w$ for BCA Algorithm. The test image is Circles(Fig\ref{fig1}(a)).}
 \label{fig7}
\end{figure}

\section{Conclusion}

In this paper, we proposed new operator-splitting algorithms for MPG noise removal. A new bilinear constraint was introduced to ensure that all corresponding subproblems can be very efficiently solved. Numerical experiment{\color{red}(s)} showed that the proposed algorithms produced comparable  results visually. Especially, compared with the TV+PD solving the same variational model, the proposed algorithms with fewer tunable parameters produced better recovery results at much faster speed. In future, we aim to analyze the  convergence of the proposed algorithm BCA$_f$, particularly without Assumption \ref{asum}, although it is empirically verifiable. In addition, we are also interested in extending the proposed algorithms to more general image reconstruction problem \cite{ding2018statistical} as well as deblurring with background noise, and we also leave it as future work.

\bibliographystyle{plain}
\bibliography{mylib}
\medskip

\end{document}